\documentclass[11pt, a4paper]{article}
\makeatletter

\newdimen\mainfontsize \mainfontsize=1\@ptsize pt
\topskip=\mainfontsize
 \@maxdepth=\maxdepth

\usepackage{psfrag}

\makeatother

\newcommand{\FL}{\mathcal{F}_L}
\newcommand{\Yi}{Y_i}

\newcommand{\AN}{A_N}
\newcommand{\ANc}{A_N^c}
\newcommand{\IAN}{1_{\AN}}
\newcommand{\IANc}{1_{\ANc}}

\newcommand{\PN}{P_N}
\newcommand{\LN}{L_N}

\newcommand{\Pl}{P^{(l)}}
\newcommand{\Bl}{B^{(l)}}
\newcommand{\El}{E^{(l)}}
\newcommand{\Elc}{E^{(l),c}}
\newcommand{\Bml}{B_m^{(l)}}
\newcommand{\Blc}{B^{(l),c}}
\newcommand{\Bmlc}{B_m^{(l),c}}
\newcommand{\Al}{A^{(l)}}
\newcommand{\Alc}{A^{(l),c}}
\newcommand{\Alminusone}{A^{(l-1)}}
\newcommand{\Alminusonec}{A^{(l-1),c}}
\newcommand{\Pk}{P^{(k)}}
\newcommand{\Po}{P^{(0)}}
\newcommand{\Pkminusone}{P^{(k-1)}}
\newcommand{\Plminusone}{P^{(l-1)}}
\newcommand{\Lkminusone}{L^{(k-1)}}
\newcommand{\Llminusone}{L^{(l-1)}}
\newcommand{\Ll}{L^{(l)}}
\newcommand{\Lk}{L^{(k)}}
\newcommand{\Lkm}{L^{(k-1)}_m}
\newcommand{\Llm}{L^{(l-1)}_m}
\newcommand{\Pkm}{P^{(k-1)}_m}
\newcommand{\Plm}{P^{(l-1)}_m}

\newcommand{\Pkbar}{\overline{P}^{(k-1)}}
\newcommand{\Plbar}{\overline{P}^{(l-1)}}

\newcommand{\eps}{\varepsilon}

\newcommand{\bn}{\mbox{$\mathbb N$}}

\newcommand{\bp}{\mbox{$\mathbb P$}}
\newcommand{\be}{\mbox{$\mathbb E$}}

\newcommand{\fm}{\mathcal{F}_L}

\newcommand{\ka}{K_j}

\newcommand{\EE}{\mathbb{E}}
\newcommand{\VV}{\mathbb{V}}

\makeatletter
\@addtoreset{equation}{section}

\makeatother

\usepackage[latin1]{inputenc}
\usepackage{amsmath}
\usepackage{amsfonts}
\usepackage{amssymb}
\usepackage{amsthm}
\usepackage{multirow}
\usepackage[dvips]{graphicx}
\usepackage{rotating}
\usepackage{verbatim}
\usepackage{appendix}
\usepackage{enumerate}
\usepackage{enumitem}

\usepackage{MnSymbol}

\title{Multilevel simulation of functionals of Bernoulli random variables with application to basket credit derivatives}

\author{K. Bujok\footnote{K. Bujok acknowledges support from EPSRC and Nomura via a CASE award, and from the Oxford-Man Institute.}, 
B. M. Hambly and C. Reisinger \vspace{0.2 cm} \\
\small Mathematical Institute, Oxford University \\ 
\small 24--29 St Giles, Oxford, OX1 3LB, UK\\ 
\small Email: \{bujok,\ hambly,\ reisinge@maths.ox.ac.uk\}}

\newtheorem{theorem}{Theorem}[section]
\newtheorem{proposition}{Proposition}[section]
\newtheorem{corollary}{Corollary}[section]
\newtheorem{rem}{Remark}[section]

\newtheorem{lemma}{Lemma}[section]

\begin{document}

\maketitle

%
%\begin{abstract}
%We study the rate of convergence of a large basket approximation, proposed in \cite{thePaper}, to price a CDO. We find that the speed of convergence is $O(1/N)$, where $N$ is a number of companies in a CDO basket, provided that the density of the portfolio loss for $N \toi$ is bounded, and that the starting values of the process of firms' \textit{distance-to-default} constitute a random sample. The numerical experiments give the same rate of convergence. Additionally, we obtain that, the empirical rate of convergence is $O(1/\sqrt{N})$ for the case when initial \textit{distance-to-default} comes from a fixed sample. Since obtaining accurate numerical results is not feasible, as we show in \cite{MyFirstPaper}, we develop a Multi-level method for increasing the number of companies in a CDO basket, based on Multi-level method in \cite{MikeMultilevel}, thus decreasing the computational time significantly. 
%\end{abstract}
%
%

%

\begin{abstract}
We consider $N$ Bernoulli random variables, which are independent conditional on a common random 
factor determining their probability distribution. We show that certain expected functionals of the 
proportion $L_N$ of variables in a given state converge at rate $1/N$ as $N\rightarrow \infty$.
Based on these results, we propose a multi-level simulation algorithm using a family of sequences with 
increasing length, to obtain estimators for these expected functionals with a mean-square error of 
$\epsilon^2$ and computational complexity of order $\epsilon^{-2}$, independent of $N$.
In particular, this optimal complexity order also holds for the infinite-dimensional limit.
Numerical examples are presented for tranche spreads of basket credit derivatives.
\end{abstract}


\section{Introduction} \label{Intro}

This article is concerned with the efficient numerical estimation of expectations of 
functionals of a large number, $N$, of exchangeable Bernoulli random variables. The objective of this work is thus two-fold:
to analyse the order of convergence in $1/N$ of expected functionals as $N$ tends to infinity, 
and to derive estimators for these expectations for which the computational
complexity is asymptotically independent of $N$.

We begin by analysing the convergence in the case of general Lipschitz and smooth functions, $p$,
of the average of $N$ exchangeable Bernoulli random variables as $N$ goes to infinity.
We then consider the case when $p$ has a certain piecewise linear structure and show that the convergence order
is the same as in the smooth case.
These results are relevant, for instance, if one wants to approximate the result for large but finite $N$ by its 
limit.
A number of applications come from the credit risk literature.
%or, conversely, approximate the infinite dimensional limit by a finite-dimensional
% system. (see, e.g., \cite{kx99}) - IS THIS REALLY A GOOD REF?. 
In \cite{Vasicek}, Vasicek derives an expression for the limiting distribution of portfolio losses in a Normal factor model,
where default of a firm is indicated by its value process being below a default barrier at maturity of the debt. In the large
portfolio limit, the randomness comes solely from a common market factor, while a law of large numbers holds for idiosyncratic
components conditionally on this factor.
Bush \emph{et al}., in \cite{thePaper}, extend this to a dynamic set-up where it is seen that the density of the limit empirical measure of firm values
satisfies a stochastic partial differential equation (SPDE) and can be used to approximate tranche spreads of basket
credit derivatives; \cite{MyFirstPaper} gives an extension to jump diffusion models while
\cite{GieseckeEtAl12} include extensions to heterogeneity and self-exciting defaults rendering the resulting equations
non-linear.
Further studies focus particularly on the tail of the limiting loss distribution, see \cite{DemboEtAl04}, \cite{GlassermanEtAl12} and the references therein.

A driving practical motivation for investigating the limiting behaviour is that the
original sequence of random variables is costly to simulate, because of the large number $N$ of underlying processes, often required over large time horizons.
Moreover, often many Monte Carlo samples are necessary for sufficiently accurate estimation of, for instance, expected tranche losses of credit basket.
This paper takes a different tack and develops a simulation method where the computational complexity is asymptotically independent of $N$.
A small tweak of the algorithm can also be used to approximate the limit obtained when $N$ goes to infinity.
%More or less as a by-product, we provide a convergence rate of expectations for increasing $N$.

%This second view raises the question of how the infinite-dimensional limit may be 
%best approximated and i

More concretely, it turns out that an interpretation of the multi-level Monte Carlo approach (see \cite{MikeMultilevel})
in the present context allows us to construct estimators based on sequences with increasing lengths and a number
of samples which decreases faster than the length increases, such that the
overall computational complexity is essentially no larger than for fixed small $N$.
%, even when computing expected functionals of arbitrarily large or infinite-dimensional systems.

A conceptually similar though distantly related approach is used in \cite{schoenmakers11}, where the multilevel idea is applied to a sequence of martingales
to estimate a dual upper bound for the value of an early exercise option. 
In that setting
they are able to show, as we do here, that the achievable complexity is
not substantially larger than that of a non-nested simulation.
The general problem of estimating conditional expectations through nested multilevel simulation is 
addressed in \cite{nanchen12}. There,  further extrapolation is used to reduce the bias of estimators, 
while here we will propose an improved estimator which reduces the variance of higher level estimators.

This article is organised as follows. In Section \ref{Setup}, we introduce the setting and outline the main 
convergence results, explaining how they can be used to construct efficient estimators.
The first key result on the convergence order of expected functionals is proved in 
Section \ref{Convergence1}, with numerical illustrations from an example of a basket credit derivative
presented in Section \ref{Numerical}. In Section \ref{Multi}, we introduce in detail two multilevel simulation 
methods and derive bounds on their computational complexity to achieve a prescribed accuracy.
Finally, in Section \ref{MultiTests} we present numerical results illustrating the efficiency gains achieved 
through multilevel simulation in this context and Section \ref{Conclusion} discusses possible extensions.

\section{Set-up and main results} \label{Setup}

In this article, we are concerned with the behaviour of ``loss'' variables describing the fraction of $N$ random 
variables in a certain state, and expected functionals of this loss variable, as $N$ goes to infinity.
The application we have in mind, and for which we will present numerical illustrations, is that of a basket of 
defaultable firms, and then the loss is the fraction of firms which default over a certain period.

More precisely, on a probability space $(\Omega,\mathcal{F}, \mathbb{P})$, consider a sequence of 
Bernoulli random variables $\Yi$, $i\in \mathbb{N}$, and a random variable $L$ taking its values in 
$[0,1]$. If required we write $\Omega=\Omega_Y\times \Omega_L$ where canonically we could 
take $\Omega_Y=\{0,1\}^{\bn}$ and $\Omega_L = [0,1]$.
The probability measure $\bp$ is constructed as follows. The random variable $L$ is generated according to its
marginal law $\bp_L$ and then, conditional 
on $\FL$, the $\sigma$-algebra generated by $L$, the $Y_i$ are independent random variables with law given by 
\begin{equation}
\label{expectL}
\bp[\Yi= 1\vert \FL] = L.
\end{equation}
To re-iterate, the Bernoulli random variables are conditionally independent given a common factor.
Thus for each $n\in \bn$
\[ \bp(Y_1 = y_1, \dots, Y_n=y_n, L\in B) = \int_B l^{s_n}(1-l)^{n-s_n} \bp_L(L\in dl), \;\;\forall y_i\in\{0,1\}, B\subset [0,1]\]
where $s_n=\sum_{i=1}^n y_i$.
We will often write $\bp_{|L}=\bp(.\mid \FL)$ for the conditional law of the $Y_i$ given $\FL$ and $\be_{|L}$ for
the associated conditional expectation.
%where $L$ is itself a random variable on $(\Omega,\mathbb{P},\mathcal{F})$ with values in [0,1],
%$\mathcal{F}_L$ the filtration generated by $L$, and $\Yi$ are independent conditional on $L$.
In the setting of defaultable firms, $\Yi = 1$ iff the $i$-th firm defaults, and $L$ is a global factor modelling 
the common tendency of firms to default. We define the loss variable to be the proportion of Bernoulli variables in state 1
\begin{equation} \label{loss}
L_N=\frac{1}{N} \sum_{i=1}^N \Yi.
\end{equation}
We consider a Lipschitz function $p$ and random variables $P$ and $P_N$ defined as
\begin{eqnarray} \label{trancheLoss0}
P&\equiv&p(L) ,\\
P_N &\equiv& p(L_N).
\label{trancheLossInfty}
\end{eqnarray}

In particular, we will study $p$ of the form
\begin{eqnarray}
\label{payoff}
p(l) \;\,\equiv\;\, [l-K_1]^{+}-[l-K_2]^{+} =
\left\{
\begin{array}{rl}
0 & l \le K_1, \\
l-K_1 & K_1 \le l \le K_2, \\
K_2-K_1 & l \ge K_2,
\end{array}
\right.
\end{eqnarray}
where $[x]^+=\max(x,0)$ denotes the positive part and $0 \le K_1 < K_2 \le 1$ are constants.
In credit derivative pricing, the particular shape of the function $p$ in (\ref{payoff}) measures the 
losses in a certain \emph{tranche} with attachment point $K_1$ and detachment point $K_2$, and its 
expectation is the building block for formulae for CDO tranche spreads.
A typical CDO pool consists of $N=125$ firms, while typical loan or mortgage books can have substantially 
more obligors, and it is therefore practically relevant to understand the behaviour of expected functionals 
for large $N$ and to devise computationally efficient estimators.

By a conditional version of the strong law of large numbers and the continuity of $p$
\begin{eqnarray} \label{lossInfty}
L_N &\rightarrow& L \quad \text{for } N \rightarrow \infty, \;\; \bp_{|L}-a.s., \\
P_N &\rightarrow& P \quad \text{for } N \rightarrow \infty, \;\; \bp_{|L}-a.s.
\label{payInfty}
\end{eqnarray}
This convergence will also hold in $L^2(\Omega_Y,\bp_{|L})$ (see Lemma~\ref{OldLemmaDwa}).

We study here the convergence rate of $P_N-P$ and
will prove the following two results. The first statement for Lipschitz and smooth functions $p$ is a relatively
straightforward consequence of (\ref{expectL}) and the easily computable $L^2$ convergence rate of $L_N$.
The second result shows that for a specific $p$ which is only piecewise smooth we can still obtain the same convergence order
as in the smooth case and with explicitly computable bounds.

\begin{theorem} \label{NewerTheorem}
Let $P$ and $P_N$ be defined by \eqref{trancheLoss0} and \eqref{trancheLossInfty}, respectively,
and assume that $p$ is Lipschitz with constant $c_p$.
We have that
\begin{eqnarray}
\label{alphapart}
\vert \be[P_N -P] \vert & \leq & \frac{c_p}{2 \sqrt{N}},\\ 
Var[P_N -P] & \leq &  \frac{c_p^2}{4N}.
\label{betapart}
\end{eqnarray}
If, moreover, $p$ is differentiable and the derivative has Lipschitz constant $C_p$, then
\begin{eqnarray}
\label{smoothest}
\vert \be[P_N -P] \vert & \leq & \frac{C_p}{8 N}.
\end{eqnarray}
\end{theorem}

\begin{theorem}
\label{OldTheorem}
For $p$ defined in (\ref{payoff}), if the cumulative density function (CDF) $F_L$ of $L$ is Lipschitz at $K_1>0$ and $K_2<1$ with Lipschitz constant $c_L$,
i.e.,
\begin{equation}
|F_L(K_j) - F_L(l)| \le c_L \, |K_j-l|
\end{equation}
for $j=1,2$ and all $l\in [0,1]$,
then 
\[ 
\vert \be[P_N -P] \vert  \leq  \frac{4c_L\sqrt{\pi}}{N}. 
\] 
\end{theorem}
Note that if $L$ has a density function which is bounded, then the CDF is certainly Lipschitz. The fact that we
only need the Lipschitz property at $K_1$ and $K_2$ will be useful for the applications considered later.

Taking the two Theorems together, order 1 for the convergence of expectations also follows for piecewise smooth
$p$ which are Lipschitz overall, provided $F_L$ is Lipschitz.

%This is close to an 
%equivalence as it is easy to see from the Radon-Nikodym Theorem and the Lebesgue Density Theorem that
%a Lipschitz CDF has a density function whose essential supremum is bounded.
%We will see from the proof that the Lipschitz property is only needed at the points $K_1$ and $K_2$, which will be
%useful for the application considered here.

These Theorems show that expected functionals for large or infinite $N$ can be successively approximated by those
with smaller $N$. Combining this with a control variate idea leads to multilevel simulation with a substantial
variance reduction for large $N$. Specifically, the above results imply that for Lipschitz $p$ we have %in the Lipschitz CDF case
$|\be [P_N - P_{M N}]| \le c_1/\sqrt{N}$ and 
$Var[P_N -P_{M N}] \le c_2/N$ for any positive integer $M$ with 
some constants $c_1$ and $c_2$.
We can consider a sequence $N_l=M^l$, $l\in \mathbb{N}$, with corresponding $\Ll=L_{N_l}$ and $\Pl=P_{N_l}$.
Translating the central idea in \cite{MikeMultilevel} to this setting, we use the decomposition
\begin{equation}
\label{telescope}
\be [\Pl] = \be [\Po] +  \sum_{k=1}^l \be [\Pk - \Pkminusone]
\end{equation}
and estimate every summand $\be [\Pk - \Pkminusone]$ separately by defining estimators
\begin{eqnarray}
\label{mlestisimple}
Z_l \equiv n_l^{-1} \sum_{j=1}^{n_l} \left( P^{(l,j)}-P_{c}^{(l,j)} \right),
\end{eqnarray}
where `c' denotes a `coarse' estimator on level $l$, i.e., using only $N_{l-1}$ instead of $N_l$ Bernoulli random variables, precisely,
\begin{eqnarray}
\label{fineP}
\label{lfine}
P^{(l,j)} &=& p(L^{(l,j)}), \quad \text{where } L^{(l,j)} = N_l^{-1} \sum_{i=1}^{N_l} Y_i^{(l,j)}, \\
P_c^{(l,j)} &=& p(L_c^{(l,j)}), \quad \text{where } L_c^{(l,j)} = N_{l-1}^{-1} \sum_{i=1}^{N_{l-1}} Y_i^{(l,j)},
\end{eqnarray}
where $Y_i^{(l,j)}$, $j=1,\ldots,n_l$, are independent samples of $Y_i$ for fixed level $l$ and independent across levels.
They are constructed from a loss factor $L^{(l,j)}$ (with the same distribution as $L$, independent across $l$ and $j$) % and independent across samples $j$ and levels $l$)
in the same way that $Y_i$ is constructed from $L$.

%with independent samples of different sizes.
The number of samples on each level, $n_l$, can be chosen to obtain an optimal allocation of computational cost for a given overall mean-square error (MSE).
The general construction in \cite{MikeMultilevel} immediately gives the following result.

\newcommand{\fracs}[2]{{ \textstyle \frac{#1}{#2} }}

\begin{proposition}[cf. \cite{MikeMultilevel}, Theorem 3.1]
\label{ml-theorem}
Let $P$, $\Pl$ as above.
If there exist independent estimators $Z_l$ based on $n_l$ Monte Carlo 
samples, and positive constants 
$\alpha, \beta, c_1, c_2, c_3$ such that 
$\alpha\!\geq\!\fracs{1}{2}$ and
\begin{enumerate}[label=\roman{*})]
\item
\label{alph}
$\displaystyle
\left| \EE[\Pl \!-\! P] \right| \leq c_1\, M^{-\alpha\, l}
$
\item
\label{cons}
$\displaystyle
\EE[Z_l] = \left\{ \begin{array}{ll}
\EE[\Po], & l=0 \\[0.1in]
\EE[\Pl \!-\! \Plminusone], & l>0
\end{array}\right.
$
\item
\label{bet}
\label{ml-iii}
$\displaystyle
\VV[Z_l] \leq c_2\, n_l^{-1} M^{-\beta\, l}
$
\item
\label{compl}
\label{ml-iv}
$\displaystyle
C_l \leq c_3\, n_l\, N_l,
$
where $C_l$ is the computational complexity of $Z_l$
\end{enumerate}
then there exists a positive constant $c_4$ such that for any $\eps \!<\! e^{-1}$
there are values $K$ and $n_l$ for which the multilevel estimator
\begin{eqnarray}
\label{multiest}
G_K = \sum_{l=0}^K Z_l,
\end{eqnarray}
has a mean-square-error with bound
\[
MSE \equiv \EE\left[ \left(G_K - E[P]\right)^2\right] < \eps^2
\]
with a computational complexity $C$ with bound
\[
C \leq \left\{\begin{array}{ll}
c_4\, \eps^{-2}              ,    & \beta>1, \\[0.1in]
c_4\, \eps^{-2} (\log \eps)^2,    & \beta=1, \\[0.1in]
c_4\, \eps^{-2-(1\!-\!\beta)/\alpha}, & 0<\beta<1.
\end{array}\right.
\]
\end{proposition}

The above result is meaningful only in situations where it is not possible or practical to sample from $L$ directly, as otherwise
$\be[P]=\be[p(L)]$ could be computed with complexity $O(\epsilon^{-2})$ in the standard Monte Carlo way.

Moreover, in some situations it is not $p(L)$ which is of interest, but $p(L_N)$ for large but finite $N$, and then it is essential to have
a method to estimate $\be[L_N]$ in a complexity which does not increase sharply in $N$.

For instance, take $N$ given and estimate $P_N$ with the standard (i.e., single level) Monte Carlo estimator
\begin{eqnarray*}
\widehat{P}_{N} &\equiv& \frac{1}{n} \sum_{j=1}^n p\left(\frac{1}{N} \sum_{i=1}^N Y_i^{(j)} \right),
\end{eqnarray*}
where $n$ is the number of samples and the $(Y_i^{(j)})$, for different $j$, are independent samples of $(Y_i)$. Then $\be[\widehat{P}_N]=\be[P_N]$ and
%\begin{eqnarray*}
$Var[\widehat{P}_{N}] = \frac{1}{n} Var[P_N]$,
%\end{eqnarray*}
where it follows from
\[
Var[P] \le 2\left( Var[P_N] + Var[P_N-P]\right),
\]
under the conditions of either Theorem \ref{NewerTheorem} or Theorem \ref{OldTheorem}, that
\begin{eqnarray*}
Var[{P}_{N}] \ge Var[P]/2 - Var[P_N-P]  \ge Var[P]/2 - c_1/N \ge c_2 Var[P],
\end{eqnarray*}
for $N$ sufficiently large and some constants $c_1, c_2$ independent of $N$. That is to say, the variance of $P_N$ and subsequently that of the estimator $\widehat{P}_{N}$ is bounded below with a positive number independent of $N$. Hence, if a MSE of $\epsilon^2$ is required for $\be[P_N]$, the complexity is
\[
C \ge n N \ge c_2 N/\epsilon^2,
\]
i.e., increases (at least) linearly in $N$. (A similar argument shows that this is also an upper bound.)

If one wants to use $\widehat{P}_{N}$ not as an estimator to $\be[P_N]$ but $\be[P]$, a bias occurs and
\begin{eqnarray*}
MSE = \be[\widehat{P}_{N} - P]^2 + Var[\widehat{P}_{N}] =
\be[{P}_{N} - P]^2 + Var[\widehat{P}_{N}] = O(N^{-2/\alpha}) + O(n^{-1}),
\end{eqnarray*}
assuming the bias is of order $\alpha$ as in Proposition \ref{ml-theorem}.
To reduce the bias and hence the error, $N$ has to be increased simultaneously with $n$.
More precisely, for MSE $\epsilon^2$ it is optimal to choose $N=O(\epsilon^{-1/\alpha})$ and $n=O(\epsilon^{-2})$, leading to a computational complexity
\[
C = O(n N) = O\left(\epsilon^{-2-1/\alpha} \right).
\]

%We point out that for fixed $N$ the standard (i.e., single level) Monte Carlo estimator has a complexity $C\le c \ \epsilon^{-2}$, however,
%$c$ increases linearly in $N$. The significance of the multilevel estimator is that the constant $c_4$ is independent of $N$.

The following Corollary addresses both cases of large finite and infinite $N$ and improves on the convergence rates of the standard Monte Carlo estimator.

\begin{corollary}
\label{complex1}
Let $P_N$ and $P$ be as in (\ref{trancheLoss0}) and (\ref{trancheLossInfty}), and assume $p$ is Lipschitz.
\begin{enumerate}
\item
There is a multilevel estimator for $\be[P]$ with MSE $\epsilon^2$ with computational complexity 
$C \le c_4 (\log\epsilon)^2 \epsilon^{-2}$.
%\end{corollary}
\item
%\begin{corollary}
For all $N$, there is a multilevel estimator for $\be[P_N]$ with MSE $\epsilon^2$ with computational complexity 
$C \le c_4 (\log\epsilon)^2 \epsilon^{-2}$, where $c_4$ is independent of $N$.
\end{enumerate}
\end{corollary}

Note that only order $1/2$ is required for the convergence of expectations in Proposition \ref{ml-theorem}, \ref{alph},
and that the complexity is then dictated by $\beta$, the case $\beta=1$ implied by Theorem \ref{NewerTheorem} for
all Lipschitz payoffs being a boundary case.

The estimators for both $\be[P]$ and $\be[P^{(L)}] = \be[P_{N_L}]$, for $N_L = M^L$ fixed, are given by (\ref{multiest}). In the first case,
the maximum level $K$ and the number of samples $n_l$ on each level have to be increased successively as part of the simulation algorithm
until a desired MSE is reached, as
explained in \cite{MikeMultilevel}.
In the second case, a similar procedure can be used but $K$ is  not increased further once the desired level $L$ is reached.
By construction, at that point, the total MSE is small enough that no additional samples need to be generated. This algorithm is formalised at the start of Section \ref{MultiTests}.

\bigskip

For the specific $p$ as in (\ref{payoff}), we can exploit the piecewise linearity of $p$ to construct multilevel estimators with even 
better complexity, by making the following observations: The summands in (\ref{telescope}) are 
unchanged if we replace $\Pkminusone = p(\Lkminusone)$ with any of $p(\Lkm)$ for  $m=1,\ldots,M$, where
\begin{eqnarray}
\label{vecest}
\Lkm \equiv \frac{1}{N_{k-1}} \sum_{i=1}^{N_{k-1}} Y_{i+(m-1) N_{k-1}}.
\end{eqnarray}
%Clearly...
This is a direct consequence of the exchangeability. Now,
\begin{eqnarray}
\label{LlMRel}
\Lk = \frac1M \sum_{m=1}^M \Lkm
\end{eqnarray}
and, if all $\Lkm$ lie in the same interval $[0,K_1]$, $(K_1,K_2]$ or $(K_2,1]$, also
$\Pk = \Pkbar$, where
\begin{eqnarray}
\label{pbar}
\Pkbar \equiv \frac1M \sum_{m=1}^M \Pkm = \frac1M \sum_{m=1}^M p(\Lkm),
 \end{eqnarray}
since $p$ is linear in these intervals.
Because of $\be[\Pkminusone] = \be[\Pkbar]$, we can now write
\begin{equation}
\label{telescope2}
\be[\Pl] = \be[\Po] +  \sum_{k=1}^l \be[\Pk - \Pkbar],
\end{equation}
and estimate the individual terms in the sum independently in the multilevel spirit, i.e., with estimators
\begin{eqnarray}
\overline{Z}_l \equiv n_l^{-1} \sum_{j=1}^{n_l} \left(P^{(l,j)}-\overline{P}^{(l,j)} \right),
\label{mlestiimpr}
\end{eqnarray}
where $P^{(l,j)}$ is defined as in (\ref{fineP}), but instead of $P^{(l,j)}_c$ we use
\begin{eqnarray}
\overline{P}^{(l,j)} &=& M^{-1} \sum_{m=1}^M p(L_m^{(l,j)}), \quad \text{where } 
\label{lscoarse}
L_m^{(l,j)} = N_{l-1}^{-1} \sum_{i=1}^{N_{l-1}} Y_{i + (m-1) N_{l-1}}^{(l,j)},
\end{eqnarray}
and where the rest of the set-up is as earlier.

There is only a variance contribution from a specific sample of the $k$-th term if at least two $\Pkm$ lie in different intervals.
For large $k$, the probability of this is small, and we will be able to show the following result.

\begin{theorem} \label{NewTheorem}
For $p$ as in (\ref{payoff}), let $\Pl$ as in Proposition \ref{ml-theorem} and $\Plbar$ as in (\ref{pbar}).
If the CDF $F_L$ of $L$ is Lipschitz with Lipschitz constant $c_L$, then
\begin{eqnarray}
%\vert E[\Pl -\Plbar] \vert & \leq & \frac{c_1}{N_l},\\ 
Var[\Pl -\Plbar] & \leq &  \frac{c_2}{N_l^{3/2}},
\end{eqnarray} 
where %$c_1=c \; M \sqrt{6 \pi} (\sqrt{2}+\sqrt{M})$, 
$c_2=c_L \; 4\sqrt{M\pi} (\sqrt{2}+\sqrt{M}) \sqrt{\frac{7}{8} (M^2+6M+1)}$.
%$c_1=...$, $c_2=...$. 
\end{theorem}

Here and throughout the paper we give explicit expressions for the constants. These should not be regarded as 
optimal in any sense.

\begin{corollary}
For Lipschitz $F_L$ and $p$ as in  (\ref{payoff}),
there is a constant $c_5$ and multilevel estimators for $\be[P]$ and $\be[P_N]$ with MSE $\epsilon^2$ with computational complexity $C\le c_5 \,\epsilon^{-2}$.
\end{corollary}
Note that we have managed to remove the logarithmic factor present in Corollary \ref{complex1} and that $c_5$ does not depend on $N$. % and now the computational complexity is asymptotically independent of $N$.

\section{Proof of convergence rates} \label{Convergence1}

We first  prove Theorem \ref{NewerTheorem} which contains statements in the general and smooth case. The rest of this section is devoted to
the proof of Theorem \ref{OldTheorem} dealing with a specific non-smooth payoff relevant to our application.
%We show a few Lemmas first.

\begin{lemma} \label{OldLemmaDwa}
Let $P_N$ and $P$ be as in (\ref{trancheLoss0}) and (\ref{trancheLossInfty}), and assume $p$ is Lipschitz with constant $c_p$. Then
\[
\be_{|L}[(P_N-P)^2] %\leq \be_{|L}[(L_N-L)^2] 
\leq \frac{c_p^2}{4N}.
\]
\end{lemma}

\begin{proof}
Since the function $p$ in \eqref{trancheLoss0} is assumed Lipschitz and $\be_{|L} [L_N] = L$,
%\begin{equation} \label{Lipschitz}
%\vert \PN - P \vert \leq \vert \LN - L \vert,
%\end{equation}
we have
\begin{eqnarray*}
\be_{|L}[(\PN-P)^2]  \leq  c_p ^2 \ \be_{|L}[(\LN-L)^2] 
= c_p^2 \ Var[\LN|\fm]
= \frac{c_p^2}{N} \ Var[\Yi|\fm] 
= \frac{c_p^2}{N} \ L(1-L).
\end{eqnarray*}
For $L \in [0,1]$, $L(1-L) \leq \frac{1}{4}$, which gives the result.
\end{proof}

\begin{proof}[Proof of Theorem~\ref{NewerTheorem}]
Equation (\ref{betapart}) follows directly from Lemma \ref{OldLemmaDwa}, and then, by Cauchy-Schwarz,
\[
\vert \be[P-P_N]  \vert \le \sqrt{ \be [\be_{|L}[(P_N-P)^2]] } \le \frac{c_p}{2 \sqrt{N}}.
\]

For differentiable $p$, we can write
\[
\be [p(L)-p(L_N)] =  \be [p'(L)(L-L_N)] + \be [r(L,L_N)],
\]
with some remainder $r$, where the first term on the left-hand side is
\[
\be [ \be_{|L}[p'(L)(L-L_N)]] = 0.
\]
If $p$ has a Lipschitz derivative,
\[
\vert p(x)-p(y) - p'(x) (x-y) \vert \le \frac{1}{2} C_p (x-y)^2
\]
for all $0\le x,y\le 1$ and the remainder term satisfies 
\[
|\be [r(L,L_N)]| \le \frac{C_p}{8N},
\]
from which (\ref{smoothest}) follows.
\end{proof}

Now, we turn to the proof of  Theorem \ref{OldTheorem} and show a few Lemmas first.
We divide the ranges of $L$ and $L_N$ into the three intervals $I_1=[0,K_1]$, $I_2=(K_2,1]$ and $I_3=(K_1,K_2]$, in each of which the function $p$ from (\ref{trancheLoss0}) is linear;
the point being that the probability of $L$ and $L_N$ lying in different intervals is small for large $N$, and the expected difference of $P-P_N$ is small if they are in the same interval. The following Lemmas quantify this.

\begin{lemma} \label{OldLemmaCztery}
For $j=1,2$, we have
\begin{eqnarray}
\bp_{|L}(L \in I_j, \LN \in I_j^c) & \leq & 1_{L \in I_j} \; e^{-N (L-K_j)^2}, \label{ldtone} \\
\bp_{|L}(L \in I_j^c, \LN \in I_j) & \leq &  1_{L \in I_j^c} \; e^{-N (L-K_j)^2}. \label{ldthtwo} 
\end{eqnarray}
\end{lemma}

\begin{proof}
This is a standard large deviations result.
By Theorem 2.2.3 in \cite{DemboZeitouni}, p.~27, and Remark (c) thereafter, for $\fm$-independent and identically distributed random 
variables $(\Yi)_{1\leq i \leq N}$ with $\be[\Yi \mid \fm]=L$, we obtain that if $0<L\leq K_j$,
\[
\bp_{|L}(L_N > \ka) \leq e^{-N g(L,K_j)},
\]
and if $\ka < L <1$,
\[
\bp_{|L}(L_N \leq \ka) \leq e^{-N g(L,K_j)},
\]
where the rate function $g(L,\ka)$ is given on p.~35 in \cite{DemboZeitouni} as
\[
g(L,\ka)=\ka \; \log\left( \frac{\ka}{L}\right)+(1-\ka)  \; \log\left( \frac{1-\ka}{1-L}\right),
\]
since $\Yi$  are Bernoulli distributed  random variables with $P(\Yi=1 \mid \fm)=L$.
It is straightforward to check that for all $L \in (0,1)$
\begin{equation} \label{sx}
g(L,\ka) \geq (\ka-L)^2.
\end{equation}
Hence, by \eqref{sx}, for $0< L \leq K_j$ 
\begin{equation} \label{upper}
\bp_{|L}(L_N > \ka) \leq e^{-N g(L,\ka)} \leq e^{-N ( \ka-L ) ^2},
\end{equation}
and similarly for $K_j < L < 1$.
%\begin{equation} \label{lower}
%\bp_{|L}(L_N \leq  \ka) \leq e^{-N g(L,\ka)} \leq e^{-N ( \ka-L ) ^2}.
%\end{equation}
These estimates are clearly true for the degenerate cases $L=0$ and $L=1$.
From this the result follows.

\end{proof}

\begin{lemma} \label{OldLemmaJeden}
Let $p$ be as in (\ref{payoff}). If
$\AN$ is the event that $L_N$ and $L$ are in the same interval and $\ANc$ its complement, then
\begin{eqnarray}
%\vert E[P_N-P]\vert &\le& 2 E[\vert E[(P_N-P) \; 1_{\ANc} \vert \FL]\vert ].
\be[(\PN-P)\IAN] = - \be[(\LN-L)\IANc 1_{\{L\in I_3\}}].
\end{eqnarray}
\end{lemma}

\begin{proof}
By splitting the range of $L$ into the different intervals,
\begin{eqnarray*}
\be[(\PN-P)\IAN] &=& \sum_{j=1}^3 \be[(\PN-P)\IAN 1_{\{L\in I_j\}}] \\
&=& \be[(\LN-L)\IAN 1_{\{L\in I_3\}}]  \\
%&=& \be[\be_{|L}[(\LN-L)\IAN 1_{\{L\in I_3\}}] \\
%&=& - \be[\be_{|L}[(\LN-L)\IANc 1_{\{L\in I_3\}}] \\
&=& - \be[(\LN-L)\IANc 1_{\{L\in I_3\}}],
\end{eqnarray*}
where we have used in the second line that $\PN=P$ if both $\LN$ and $L$ lie in either $I_1$ or $I_2$ and that $\PN-P=\LN-L$ in $I_3$;
in the last line that $\be_{|L}[\LN-L] = 0$ and $\IAN + \IANc =1$.
\end{proof}

\begin{lemma} \label{OldLemmaPiecA}
Let $A_N^c$ be as in Lemma \ref{OldLemmaJeden}. 
If the CDF $F_L$ of $L$ is Lipschitz at $K_j$, $j=1,2$, with constant $c_L$,
%and $K_2$ with Lipschitz constant $c$,
%, such that for $l \in [0,1]$, $f_L(l) \leq c$, where $c$ is a positive constant,
then 
\begin{eqnarray}
\be \left[ \left( \bp_{|L}[{\ANc}] \right)^{\frac{1}{2}} \right] \leq \frac{c_L \; 4 \sqrt{\pi}}{\sqrt{N}}.
\end{eqnarray}
\end{lemma}

\begin{proof}
Let $I_1^{c}=(K_1,1]$ and $I_2^{c}=[0,K_2]$ be the complements in $[0,1]$ of  $I_1$ and $I_2$, then
\[
\ANc \subseteq \{L \in I_1,  \LN \in I_1^c \} \cup \{L \in I_2,  \LN \in I_2^c \} \cup
\{L \in I_1^c,  \LN \in I_1 \} \cup \{L \in I_2^c,  \LN \in I_1^c \}
\]
and therefore
\begin{eqnarray}
\bp_{|L}[{\ANc}] & \leq & \bp_{|L}[ {L \in I_1},   \; {\LN \in I_1^c}]
+ \bp_{|L}[ {L \in I_2},   \; {L_N \in I_2^c}]
+ \bp_{|L}[ {L \in I_1^c}, \; {L_N \in I_1}] \nonumber \\
&& + \; \bp_{|L}[ {L \in I_2^c}, \; {L_N \in I_2}]. \nonumber
\end{eqnarray}
By \eqref{ldtone}, \eqref{ldthtwo} we have %, and also since $1_{L \in I_j} \leq 1$, $1_{L \in I_j^c} \leq 1$, we have
\begin{eqnarray*}
\bp_{|L} [{\ANc}] \leq 2 \; \left (e^{-N (L-K_1)^2}+ e^{-N (L-K_2)^2} \right), 
\end{eqnarray*}
and 
%from
%\begin{eqnarray}
%\left( \bp_{|L}[{\ANc}] \right)^{\frac{1}{2}} & \leq & 2^{\frac{1}{2}} \; \left (e^{-N (L-K_1)^2}+ e^{-N (L-K_2)^2} \right)^{\frac{1}{2}} \nonumber \\
%& \leq & 2^{\frac{1}{2}} \; \left (e^{-N \frac{(L-K_1)^2}{2}}+ e^{-N \frac{(L-K_2)^2}{2}} \right), \nonumber 
%\end{eqnarray}
we obtain
\begin{eqnarray}
\be \left[ \left( \bp_{|L}[{\ANc}] \right)^{\frac{1}{2}} \right] 
& \leq & 2^{\frac{1}{2}} \; \left (\be  \left[ e^{-N \frac{(L-K_1)^2}{2}} \right] + \be \left[ e^{-N \frac{(L-K_2)^2}{2}} \right]  \right). %\nonumber 
\label{beforeint}
\end{eqnarray}
If we extended $F_L$ by 0 and 1 from $[0,1]$ to $\mathbb{R}$ then, for $j=1,2$, we have
\begin{eqnarray}
\be \left[ e^{-N \frac{(L-K_j)^2}{2}} \right] &=& 
\int_{-\infty}^{\infty} e^{-N \frac{(l-K_j)^2}{2}} \;  dF_L(l) \\
&=& N\; \int_{-\infty}^\infty (l-K_j) e^{-N \frac{(l-K_j)^2}{2}} F_L(l) \; dl \label{lipuse} \\
&\leq& N F_L(K_j) \int_{-\infty}^{\infty} (l-K_j) e^{-N \frac{(l-K_j)^2}{2}} \; dl +
c_L \; N \int_{-\infty}^{\infty} (l-K_j)^2 e^{-N \frac{(l-K_j)^2}{2}} \; dl \nonumber \\
& = &  \frac{c_L \sqrt{2\pi}}{\sqrt{N}}, \nonumber
\end{eqnarray} 
where we used the Lipschitz property of the CDF after (\ref{lipuse}) and then integrated exactly.
%we use that $f_L(l) \leq c$, and we change variables with $y=(l-K_j) \sqrt{N}$. 
The result follows directly by insertion in (\ref{beforeint}).
\end{proof}

\begin{proof}[Proof of Theorem~\ref{OldTheorem}]
By the tower property of conditional expectations and Jensen's inequality, we have
\begin{eqnarray}
\vert \be [ (\PN-P)  \; 1_{\ANc}] \vert \leq  \be [ \vert \be_{|L} [ (\PN-P)  \; 1_{\ANc}] \vert]. 
\end{eqnarray} 
Then Cauchy-Schwarz %to $\vert E[ (P_N-P)  \; 1_{A^N_2} \mid \fm] \vert$, 
gives
\begin{eqnarray} \label{above}
\vert \be [ (\PN-P)  \; 1_{\ANc}] \vert \leq \be_{|L} \left[ \left(\be [(\PN-P)^2]\right)^{\frac{1}{2}}
\; \left(\bp_{|L}[{\ANc}] \right)^{\frac{1}{2}} \right].
\end{eqnarray}
By Lemmas \ref{OldLemmaDwa} and \ref{OldLemmaPiecA}, we obtain %for Lipschitz $F_L$ that
\begin{eqnarray}
\vert \be[(\PN-P) \; 1_{\ANc}] \vert \leq \frac{c_L 2\sqrt{\pi}}{N}.
\end{eqnarray}

Similarly, using Lemma \ref{OldLemmaJeden} and the same argument as above,
\[
\vert \be[(\PN-P) \; 1_{\AN}] \vert \leq \frac{c_L 2 \sqrt{\pi}}{N},
\]
from which the statement follows.

%For non-Lipschitz $F_L$, clearly
%\[
% \be \left[ \left(\bp_{|L}[{\ANc}] \right)^{\frac{1}{2}}\right] \le 1
%\]
%and we get the statement immediately.
%
%%Let us now find the upper bound
%
%The upper bound on $Var[P_N-P]$
%follows immediately from Lemma \ref{OldLemmaDwa}.
\end{proof}

\section{An application and numerical results} \label{Numerical}

To illustrate the theoretical rate of convergence,
we study numerical results for expected tranche losses of a synthetic CDO for an increasing size $N$ of the underlying CDS pool.

We consider a structural factor model (see, e.g., \cite{lipton02,MyFirstPaper}), where the \emph{distance-to-default} of the $i$-th firm,
$i=1,\ldots,N$, evolves according to
\begin{eqnarray}
\label{Xti}
X_t^i &=& X_0^i+\beta t+\sqrt{1-\rho} \; W_t^i+ \sqrt{\rho} \; B_t + J_t , \quad t>0,
\end{eqnarray}
where $\rho \in [0,1)$, $\beta$ given. %=(r-\lambda\sigma \nu - \sigma^2/2)/\sigma$ with $r, \nu, \sigma>0$,
Here, $B$ is assumed to be a standard Brownian motion and $J_t = \sum_{k=1}^{{C\!P}_t} \Pi_k$, where ${C\!P}_t$ is a Poisson process with intensity $\lambda$ and $\Pi_k$ are independent Normals with mean $\mu_{\Pi}$
and variance $\sigma_\Pi^2$, while all $W^i$ are independent standard Brownian motions and independent of $B$ and $J$. Thus $B$ and $J$ model factors affecting the whole market, whereas $W^i$ are idiosyncratic effects.
%Dependence is modelled by exposure to common factors $B$ and $J$,

The $i$-th firm is considered to be in default if its distance-to-default %, $X_t^i$,
is below 0 at any one of the observation times $T_j = j q$, $q=0.25$ (quarterly), up to $T_{20}=T=5$, the assumed maturity of the debt here.
We introduce the default time $\tau_i$ and Bernoulli random variable $\Yi$ indicating default of the $i$-th firm before $T$, by
\begin{eqnarray}
\nonumber
\tau^i &=& \inf \left( \{t \in \{T_1,\ldots, T_M\}: X_t^i \le 0 \} \cup \{\infty\} \right), \\
\Yi &=& 1_{\{\tau^i\le T\}}.
\label{Yi}
\end{eqnarray}

%such that $\langle\text{d}X_t^i, \text{d}X_t^j \rangle = \rho \, \text{d}t$ for $i\neq j$.

For the numerical experiments, the initial values $X_0^i$ are %either fixed at $\mu_{X_0}=4.6$ for all $i$ or 
drawn independently from a Normal distribution, %with mean $\mu_{X_0}$ and variance $\sigma^2_{X_0}$,
\begin{eqnarray*}
X_0^i &\sim& N(\mu_{X_0},\sigma^2_{X_0}),
\end{eqnarray*}
where the mean $\mu_{X_0}=4.6$ and standard deviation $\sigma_{X_0}=0.8$ are obtained from a calibration to iTraxx data as detailed in \cite{MyFirstPaper},
as
%The other parameters used 
are $\rho=0.13$, % and $\beta=??$. % is adjusted
%, $\sigma=0.35$, 
$\lambda=0.04$, $\mu_{\Pi} = -0.5$ and $\sigma_\Pi^2=0.17$. %, see also \cite{MyFirstPaper}. % and $\nu$ 
%such that it makes the discounted asset price process a martingale.

That the definition of $Y_i$ in (\ref{Yi}) fits into the initial set-up is a consequence of the exchangeability of $X_t^i$ in (\ref{Xti}).
If we define %$\overline{X}_t^i$ by
\begin{eqnarray*}
\overline{X}_t^i &\equiv& \left\{
\begin{array}{rl}
X_t^i, & \quad t < \tau^i, \\
0, & \quad t\ge \tau^i,
\end{array}
\right.
\end{eqnarray*}
then the $\overline{X}_T^i$ are still exchangeable.
Hence, by de Finetti's Theorem (see \cite{Kallenberg}),
there exists a random measure $\alpha$ on $\mathbb{R}$ such that a.s.
\[
\alpha(B) = \lim_{N\rightarrow\infty} \frac{1}{N} \sum_{i=1}^N 1_{\overline{X}_T^i \in B}
\]
for all Borel sets $B$.
Conditional on $\alpha$, the $\overline{X}_T^i$ and $Y_i$ are i.i.d.
The link to the random variable $L$ is established by defining %setting $B=\{0\}$, and then
\[
L \equiv \alpha(\{0\}) = \lim_{N\rightarrow\infty}  \frac{1}{N} \sum_{i=1}^N 1_{\{\overline{X}_T^i=0\}} = \lim_{N\rightarrow\infty}  \frac{1}{N}\sum_{i=1}^N Y_i
= \lim_{N\rightarrow\infty} L_N.
\]
Clearly, $Y_i$ takes values in $\{0,1\}$ and $\bp[Y_i=1|\FL] = \be_{|L}[Y_i] = \be_{|L}[L_N] = L$.

It is shown in \cite{MyFirstPaper} that the above random measure $\alpha$ is the sum of $L$ times a Dirac measure located at 0 and a continuous part which 
satisfies a stochastic partial differential equation. % (see \cite{MyFirstPaper} for details).
To generated (approximate) samples of $L$,  we numerically solve the SPDE by a combined Monte Carlo finite difference method (see again \cite{MyFirstPaper})
to generate samples of the random measure, and use this to compute $L$.
So, on this instance, there is an alternative -- albeit very costly -- way of simulating $L$ directly, and we use this to investigate the relevant properties of $L$ empirically.

Specifically, in view of the conditions of Theorem \ref{OldTheorem}, we illustrate the numerically computed CDF $F_L$ of $L$ for different parameters in Figure \ref{fig_EmpiricalCDF_Final}.
%\begin{sidewaysfigure}
%%\centering
%\scalebox{0.5}
%{\hspace{-5.5 cm} \includegraphics{EmpiricalCDF_Final.eps}}
%\caption{Empirical CDF $F_L$ for different values of $\mu_0$ (top) and different $\rho$ (bottom). All other parameters are fixed as given in the text.
%The second and third columns are zoomed into the ranges of $L$ close to 0 and 1.}
%\label{fig_EmpiricalCDF_Final}
%\end{sidewaysfigure}
\begin{figure}
\centering
\scalebox{0.65}
{\hspace{-3.5 cm} \includegraphics{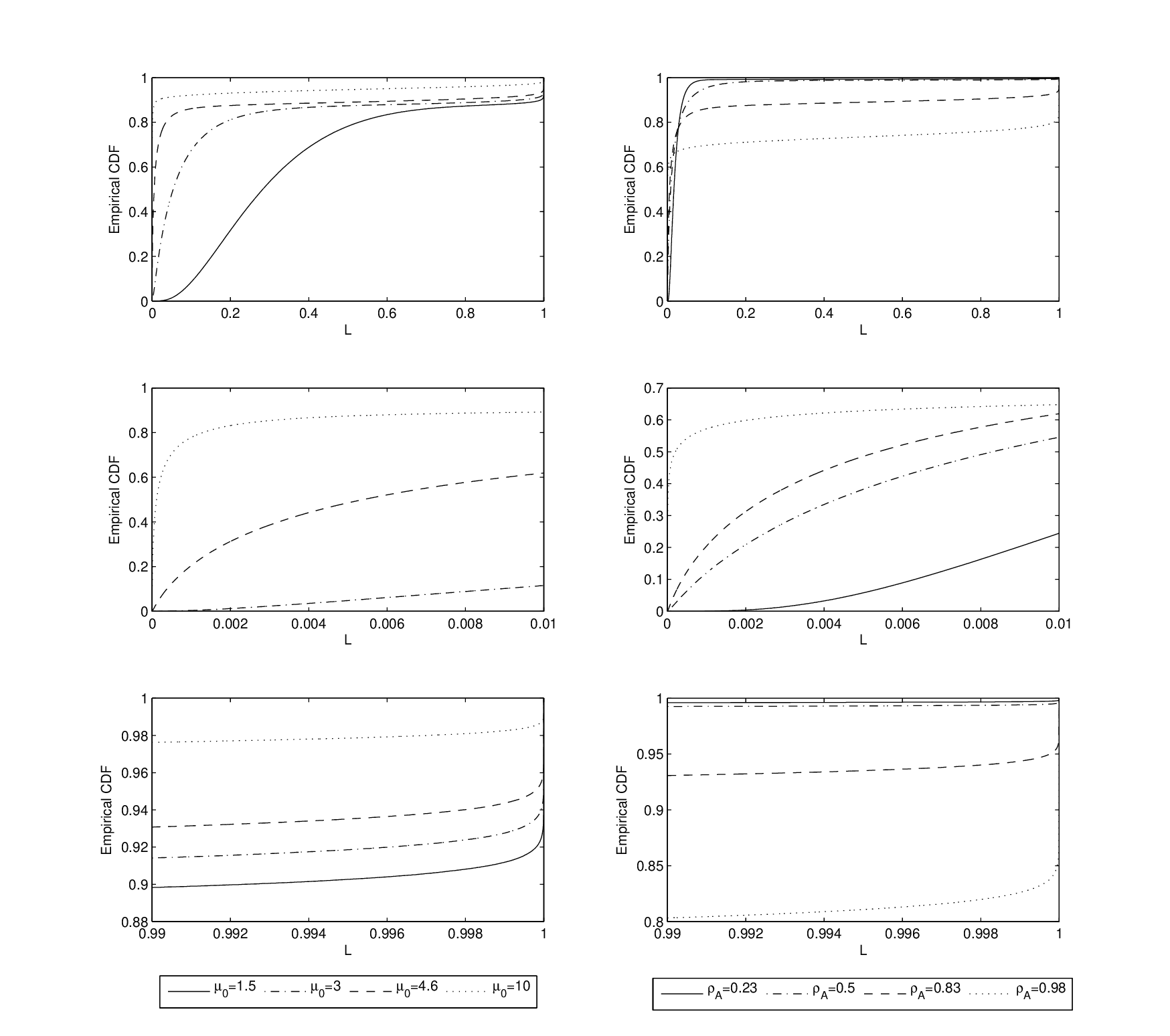} \vspace{0.6 cm}}
\caption{Top row: Empirical CDF $F_L$ for different values of $\mu_0= \mu_{X_0}$ (left) and different $\rho_A$ (right).
The values of $\rho_A$ are arrived at by (\ref{rhoA}) from $(\rho,\lambda) \in \{
(0.03, 0.001),
(0.1, 0.002),
(0.35, 0.0035),
(0.35, 0.0351),
(0.8 , 0.1)
\}$.
All other parameters are fixed as given in the text.
The plots in the second and third rows are zoomed into the ranges of $L$ close to 0 and 1, respectively.}
\label{fig_EmpiricalCDF_Final}
\end{figure}
It appears that $F_L$ is Lipschitz in $(0,1)$  but that the derivative at 0 and 1 can become very large in certain parameter ranges for $\mu_0 = \mu_{X_0}$ and overall instantaneous correlation
\begin{eqnarray}
\label{rhoA}
\rho_A=(\rho+\zeta)/(1+\zeta), \quad \zeta=\lambda(\mu_\Pi^2 + \sigma_\Pi^2),
\end{eqnarray}
between $X_t^i$ and $X_t^j$ (see \cite{MyFirstPaper}).

For large values of $\mu_0$, the probability of defaults becomes very small and the density of $L$ is concentrated around 0.
For $\rho_A$ approaching 1, all $Y_i$ become identical and therefore either all or none of the firms default, such that here the density of $L$ is concentrated at 0 and 1.
In the degenerate case $\rho_A=0$ (i.e., $\rho=\lambda=0$), $L$ is deterministic, the measure is atomic and $F_L$ a step function.

The empirical evidence thus suggests that $F_L$ is Lipschitz in the range $(0,1)$. Given that Theorem \ref{OldTheorem}
only requires the Lipschitz property at interior values $K_j$, the conditions appear to be satisfied and the Theorem to apply in this setting.
Even in situations where $F_L$ has a bounded derivative at 0 and 1, the fact that only the Lipschitz constants from $K_1$ and $K_2$ enter into the estimates gives us substantially smaller bounds.

We now move on to present numerical results for the payoff function $p$ from (\ref{payoff}) illustrating the convergence as the number of firms $N$ goes to infinity.
We consider portfolios consisting of $N_k=M^k=5^k$ companies for $k=1,\ldots,7$. % or $K=7$ for different examples.

To include a recovery value of defaulted firms in the model, we rescale $L_N$ by $(1-R)$, where $R=0.4$ is the recovery rate.
Equivalently, we pick $(K_1,K_2)=(1-R)^{-1} (a,d)$ in (\ref{trancheLoss0}) and $(a,d) \in \{(0,0.03),(0.03,0.06),(0.06,0.09),(0.09,0.12),(0.12,0.22),(0.22,1)\}$ 
as the attachment and detachment points for iTraxx tranches, and then study $(1-R) p(L_N)$. 

A straightforward Monte Carlo estimator for expected tranche losses $\be [\Pk]$ is then given by
\begin{eqnarray}
\widehat{G}_k &=& \frac{1}{n} \sum_{j=1}^n (1-R) p(L^{(k,j)}), \\
L^{(k,j)} &=& \frac{1}{N_k} \sum_{i=1}^{N_k} Y_i^{(j)},
\end{eqnarray}
where $(Y_i^{(j)})$ are independent samples of $Y_i$,
%are identically distributed to $(Y_i)$, and for fixed $i$ are independent for different $j$, 
i.e.,
%the estimator $\widehat{G}_k$ is the average over $n$ samples 
corresponding to independent paths for $B$, $W$ and $J$.
There is no time discretisation error as (\ref{Xti}) can be sampled directly.
However, it turns out to be computationally prohibitively expensive to choose $n$, the number of samples, large enough to produce estimators with sufficiently small RMSE to allow us to distinguish between $\widehat{G}_k$ and
$\widehat{G}_{k+1}$ for large $k$. 

We therefore use the multilevel simulation approach outlined in Section \ref{Setup} and detailed further in Section \ref{Multi}.
The point is that the differences $G_{k+1}-G_k$ are simulated directly in the multilevel approach.
Therefore, 
we approximate $|G-G_k|$, where $G\equiv \lim_{k\rightarrow\infty} G_k$, by
\begin{equation} \label{Zk}
S_k=\vert G_k-G_K \vert = \Bigg\vert \sum_{l=k+1}^K Z_l \Bigg\vert
\end{equation}
for $k<K$, where $Z_l$ is an estimator for $\be[\Pl-\Plminusone]$ as used in the construction of $G_k$ in (\ref{multiest}) (precisely, we used the estimator $Z_l$ defined later in (\ref{mlestisimple})).
The difference between $S_k$ and $|G-G_k|$ for $k=K-1$ is given by $G_{K-1}-G_K \approx (G_{K-1}-G)(1-1/M)$ and for $k=K-2$ by $G_{K-2}-G_K \approx (G_{K-2}-G)(1-1/M^2)$.
Given $M=5$ in our examples, the error due to this approximation will be seen to be smaller than the estimation error.

The results are shown in Figure \ref{fig_ConvergenceX0_resampling_Improved}.
%for deterministic and normally distributed starting points, resp.
We plot the logarithm of $S_k$ to base $M$, together with the sample standard deviation of the the multilevel estimators $G_k$ (see (\ref{multiest})) and
\begin{equation} \label{yk}
y_k= -k + y_0,
\end{equation}
where $y_0$ is a suitably chosen constant, to verify the predicted convergence order empirically.
\begin{figure}
\centering
\scalebox{0.75}
{\hspace{-2.2 cm} \includegraphics{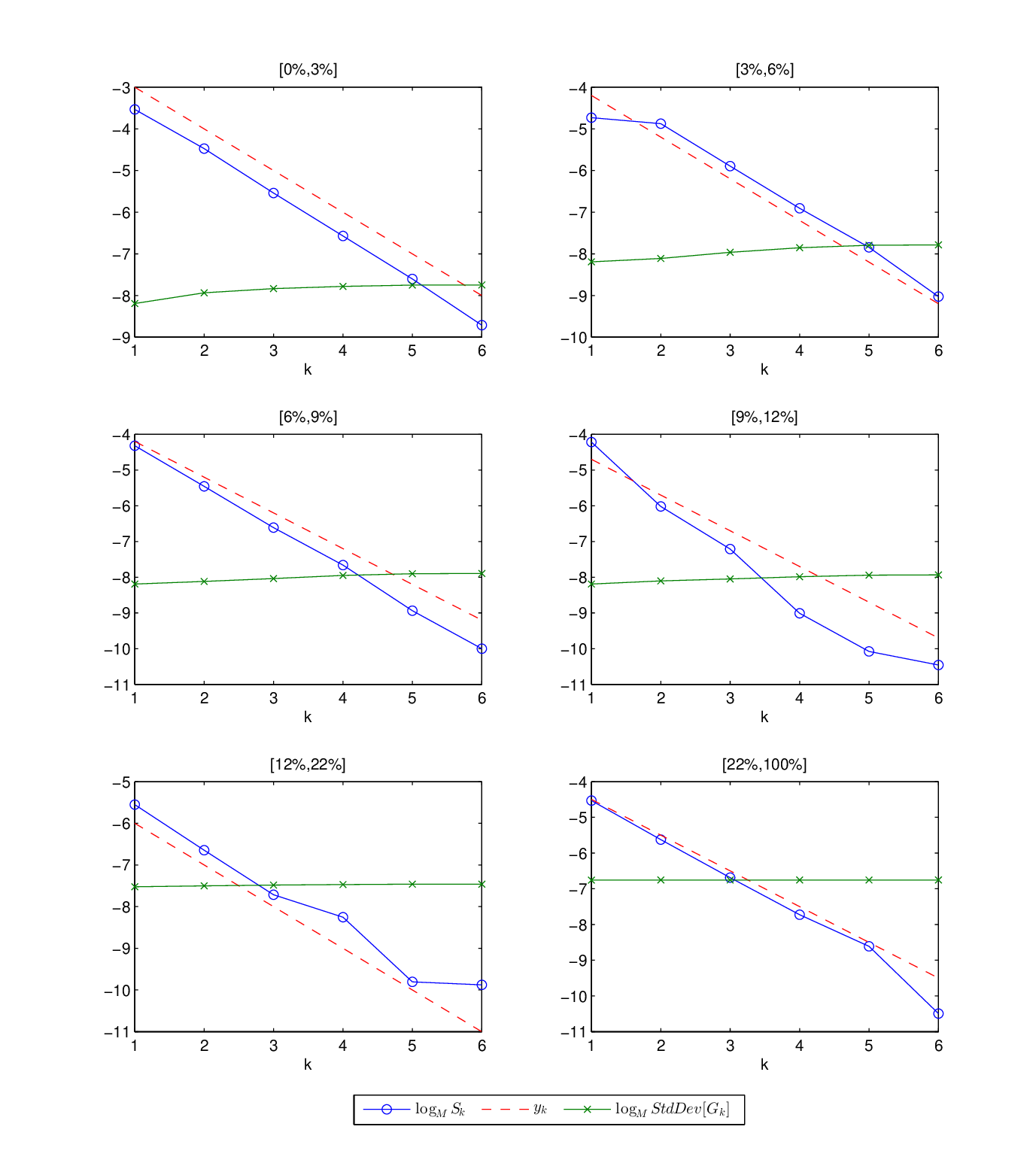}}
\caption{
Shown here is $\log_M S_k$, where $S_k$ given by \eqref{Zk} is an estimator for $|\be[\Pk-P]|$.
The various plots are for tranches ranging from [0\%-3\%] to [22\%-100\%], 
of a CDO basket consisting of $N_k=M^k=5^k$ companies, where $k=1, \ldots, 6$.
%The rate of convergence of $\be[\Pk]$ to $\be[P]$, illustrated by , together with a fitted 
The comparison with the predicted trend $y_k$ from \eqref{yk}
confirms the first order convergence.
%, where $\hat{\alpha}=1$,
Included is also the standard deviation of the estimated tranche loss $G_k$. %, calculated
%Moreover, $X_0^i \sim N(\mu_{X_0}, \sigma_{X_0})$ $\forall i$,
%where  $\mu_{X_0}=4.6$, $\sigma_{X_0}=0.8$,
%and $X_0^i$ is resampled for each simulation of $(X_t^i)_{0 \leq t \leq T}$.}
%As Figure \ref{fig_ConvergenceX0const},
%but here $X_0^i \sim N(\mu_{X_0}, \sigma_{X_0})$ $\forall i$,
%where  $\mu_{X_0}=4.6$, $\sigma_{X_0}=0.8$,
%and $X_0^i$ is resampled for each simulation of $(X_t^i)_{0 \leq t \leq T}$, and $K=7$.
}
\label{fig_ConvergenceX0_resampling_Improved}
\end{figure}
The data points appear to be in good agreement with first order convergence.

\section{Analysis of the multilevel method} \label{Multi}

In this section, we describe and analyse a multilevel simulation approach for the estimation of expected functionals 
of the form (\ref{trancheLoss0}) and (\ref{trancheLossInfty}), the latter with a particular emphasis on the case of large $N$. % and, specifically, in the large basket limit.

%\subsection{The basic method}
\label{subsec:basic}

The multilevel Monte Carlo method proposed by Giles in \cite{MikeMultilevel} estimates the expected value of a functional of the solution to a stochastic differential equation obtained by a timestepping scheme. It performs computations on different refinement levels $l$ with time steps $h_l = h_0 M^{-l}$ for $M>1$, such as to
minimise the overall computational time of the Monte Carlo estimator for prescribed mean square error (MSE).
Since the MSE consists of a Monte Carlo error (variance) and a discretisation error (bias), the method controls both the number of samples $n_l$ on level $l$, 
to bound the Monte Carlo variance of order $O(n_l^{-1})$, and the finest $L$ with time step $h^{-L}$ on which to approximate the SDE, in order to reduce the bias.
The multilevel method is based on two premises: Monte Carlo estimators for an increasing number of time steps converge at a certain order in $h_l$, and the computational cost needed to calculate an estimator increases with $n_l h_l^{-1}$.
In this approach, estimators obtained with a smaller number of time steps are used as control variates for estimators with a larger number of time steps, which significantly decreases the computation time. 
%What is more, as a by-product of the method, some of the Monte Carlo estimators of the function calculated with the number of time steps, lower or equal to the one giving the chosen level of MSE, are obtained.

%The present analogue to $h_l$ is $N_l^{-1}$, the inverse of the number of random variables. 

To obtain a complexity result for an estimator of $\be[P]$ with $P$ from (\ref{trancheLoss0}), we %translate the results of \cite{MikeMultilevel} into this language and immediately obtain Proposition \ref{ml-theorem} from Section \ref{Setup}. 
substitute $h_l$ by $N_l^{-1}$ in Theorem 3.1 of \cite{MikeMultilevel} and immediately obtain Proposition \ref{ml-theorem} from Section \ref{Setup}.

%
%
%and where $Y_j^{(l,i)}$ are independent for all $l, i$, identically distributed to $Y_j$ such that
%they are independent across $j$ (for the same $l$, $i$) conditionally on the random variable $L^{(l,i)}$,
%which is also independent for different $(l,i)$ and identically distributed to $L$.
%BIT CONVOLUTED -- REPHRASE

By direct inspection, for the construction of $Z_l$ from (\ref{mlestisimple}), Assumption \ref{cons} holds in Proposition \ref{ml-theorem}.
From Theorem \ref{NewerTheorem}, we know that \ref{alph} holds with $\alpha=1/2$ for general Lipschitz $p$.
Clearly, the computational effort to compute $Z_l$ is proportional to $n_l N_l$ as required in \ref{compl}.
Finally, \ref{bet} holds by the following simple application of Lemma \ref{OldLemmaDwa}.
\begin{proposition} \label{ThmPl}
Let $\Pl=P_{N_l}$ as per (\ref{trancheLossInfty}), %and $P$ by \eqref{PlMDefinition}, 
where $p$ is Lipschitz with constant $c_p$, then
\begin{eqnarray}
\label{basicvar}
%\vert E[P_l-P^{m}_{l-1}] \vert & \leq & \frac{C_1}{N_l}, \\
Var[\Pl-\Plminusone] & \leq & c_p^2 \ \frac{M+1}{2N_l}.
\end{eqnarray} 
\end{proposition}
%
%*****
%
%Note we could `improve' this constant to $(1+\sqrt{M})^2/4$. The improvement being in the constant for the 
%asymptotics in $M$. Do the values of these and other constants matter?
%
%*****

\begin{proof}

This follows directly from
\begin{eqnarray}
\be_{|L}[(\Pl-\Plminusone)^{2}] &=& \be_{|L}[((\Pl-P)-(\Plminusone-P))^{2}] \nonumber \\
&\leq& 2 \left( \be_{|L}[(\Pl-P)^{2}] + \be_{|L}[(\Plminusone-P)^{2}]   \right)
\end{eqnarray}
by Lemma \ref{OldLemmaDwa} and taking expectations over $L$.
\end{proof}
\noindent
We have therefore proven the first statement of Corollary \ref{complex1}. \\

In practice, it is also relevant to be able to compute $\be[P_N]$ efficiently for finite $N$. It is clear that for fixed $N$, the complexity is bounded by $c \ \epsilon^{-2}$ for some $c>0$,
but for a na\"{\i}ve (single-level) estimator the constant $c$ will increase with $N$, as detailed in Section \ref{Setup}.
From the proof of Theorem 3.1 in \cite{MikeMultilevel} it is clear, however, that there is a multilevel estimator with \emph{a priori} bounded upper level $K$ which satisfies
the second statement in Corollary \ref{complex1}. \\

%\subsection{Improved multilevel estimator}
\label{subsec:improved}

We now discuss the multilevel estimator $\overline{Z}_l$, based on the faster decay rate 3/2  for piecewise linear payoffs
in Theorem \ref{NewTheorem}, which we prove subsequently.

It is clear that $\overline{Z}_l$ satisfies \ref{cons} in Proposition \ref{ml-theorem} and that the computational complexity is still bounded as required per \ref{compl}.
In fact, as the main computational cost is typically in sampling $\Yi$, the computational complexity is virtually identical to that of $Z_l$.
In particular, if we evaluate (\ref{lfine}) by using (\ref{LlMRel}) and the already computed (\ref{lscoarse}), 
the difference in evaluating $Z_l$ and $\overline{Z}_l$ is an $O(M)$ cost, i.e., independent of $N_l$.
Now, given Theorem \ref{NewTheorem}, we have that
\begin{eqnarray}
\label{varbar}
Var[\overline{Z}_l] \le c \, n_l^{-1} M^{-3/2 \, l},
\end{eqnarray}
for some $c$, such that we are in the first regime in the complexity result of Proposition \ref{ml-theorem}, i.e., we have optimal complexity order. \\

We have not commented so far on the (optimal) selection of $M$. The choice of $M=5$ in Section \ref{Numerical} was to some extent dictated by the application of a CDO basket where the target size is $N=125=5^3$, and therefore for $M=5$ this $N$ is reached exactly for level $K=3$.
For different $M$, or indeed for $N$ which is not an integer power of an integer $M$, one can adapt the method easily by choosing
$K$ as the largest integer such that $M^{K-1} < N$, and then estimate the correction between $\be[P^{(K-1)}]$ and $\be[P_N]$
by a last estimator $Z_K$. Such considerations are obviously irrelevant for the estimation of $\be[P]$, and there the choice of $M$ is entirely dictated by complexity issues.

The total error is a combination of the bias, dictated by the number of Bernoulli random variables $N_K$ on the finest level and therefore largely independent of $L$, and the variance of the individual estimators $Z_l$ or $\overline{Z}_l$. The effect of increasing $M$ is that the variance of $Z_l$ may increase, but conversely the number of levels required to reach a given $N$ will decrease and therefore the total number of random variables which need to be simulated may be lower. There is a discussion in \cite{MikeMultilevel} on the optimal selection, with a heuristic calculation for $\beta=1$, suggesting an optimal value of 6 or 7, which is then lowered to 4 in computations to incorporate a sufficient number of levels for a reliable estimation of the variance on course levels. For a faster decay of the variance, $\beta=3/2$, the optimal $M$ can be expected to be smaller, and therefore $M=5$ seems a sensible choice, although we did not test this systematically.

The remainder of this section is devoted to the proof of Theorem \ref{NewTheorem}.

\begin{lemma} \label{thmb2}
Assume the CDF $F_L$ of $L$ is Lipschitz with constant $c_L$.
Let $\Bl$ be the event that $\Ll$ lies in the same interval as $\Llminusone\!$, $\Blc$ its complement, then
\begin{eqnarray}
\be\left[\left( \bp_{|L}[{\Blc}] \right)^{\frac{1}{2}}\right] \leq \frac{C}{\sqrt{N_l}}, \nonumber 
\end{eqnarray}
where $C=c_L \; 4 \sqrt{\pi} (\sqrt{2}+\sqrt{M})$.
\end{lemma}

\begin{proof}
Let $\Al$ again be the event that $\Ll$ and $L$ are in the same interval,
$\Alc$ its complement.
%Using notation from Lemma \ref{thm2extra}, event $B^m_2$ can be written as
Then from
%\begin{eqnarray}
%B^m_2 &=& \left( A_1^{l} \cap A_2^{l-1,m} \cap B^m_2 \right) \cup \left( A_2^{l} \cap A_1^{l-1,m} \cap B^m_2 \right) \cup \left( A_2^{l} \cap A_2^{l-1,m} \cap B^m_2 \right), \nonumber
%\end{eqnarray}
%which gives
\begin{eqnarray}
\Blc &\subseteq& \left( \Al \cap \Alminusonec \right) \bigcup \left( \Alc \cap \Alminusone \right) \bigcup \left( \Alc \cap \Alminusonec  \right) \nonumber
\end{eqnarray}
%\begin{eqnarray}
%1_{B^m_2} &=& 1_{A_1^l} \; 1_{A_2^{l-1,m}} + 1_{A_2^l} \; 1_{A_1^{l-1,m}} + 1_{A_2^l} \; 1_{A_2^{l-1,m}} \; 1_{B^m_2}, \nonumber
%\end{eqnarray}
follows
\begin{eqnarray}
\bp_{|L}[{\Blc}] &\leq & \bp_{|L}[{\Al} \cap {\Alminusonec}] + \bp_{|L}[{\Alc} \cap {\Alminusone}] + \bp_{|L}[{\Alc} \cap {\Alminusonec}] \nonumber \\
&\leq& 2 \; \bp_{|L}[{\Alc}]  + \bp_{|L}[{\Alminusonec}], \nonumber
\end{eqnarray}
which leads to
\begin{eqnarray}
\be \left[ \left( \bp_{|L}[{\Blc}] \right)^{\frac{1}{2}} \right] \leq \sqrt{2} \; \be \left[ \left( \bp_{|L}[{\Alc}] \right)^{\frac{1}{2}} \right]  + \be \left[ \left(  \bp_{|L}[{\Alminusonec}] \right)^{\frac{1}{2}} \right]. \nonumber
\end{eqnarray}
By Lemma \ref{OldLemmaPiecA}, we obtain the result.
\end{proof}

\begin{lemma} 
\label{thm0} 
\label{nowelemmaone}
\label{nowelemma}
For $P$, $\Pl$, $\Plminusone$ and $\Plbar$ as above, $p$ Lipschitz with constant 1,
\begin{eqnarray}
\be_{|L}[(\Pl-P)^4] & \leq & \frac{3}{16N_l^2}\left(1+\frac{4}{3N_l}\right) \leq \frac{7}{16N_l^2}, \label{fourtha} \\
\be_{|L}[(\Pl-\Plminusone)^4] & \leq & \frac{C}{N_l^2}, \label{fourthb} \\
\be_{|L}[(\Pl-\Plbar)^4] & \leq & \frac{C}{N_l^2}, \label{fourthc}
\end{eqnarray}
 %, $C_2=\frac{13 M^2}{16}$, $m=1,\ldots,M$.
where $C = \frac{7}{8} (M^2+6M+1)$.
% and$C_3=\frac{13}{8} (M^2+6M+1)$.
\end{lemma}
\begin{proof}
See Appendix \ref{app:moments}.
\end{proof}

%\begin{lemma} \label{thm4}\label{lastlemma}
%Assume the CDF $F_L$ of $L$ is Lipschitz with constant $c_L$.
%Let $\El$ be the event that all $\Llm$ lie in the same interval, $1\le m\le M$, and $\Elc$ its complement, then
%for $p$ given by (\ref{payoff})
%\begin{eqnarray}
%\label{secmomimp}
%\be[(\Pl-\Plbar)^2] \;=\; \be[(\Pl-\Plbar)^2 1_{\Elc}]  \;\leq\; \frac{C}{\sqrt{N_l}N_l}, 
%\end{eqnarray}
%where $C=c_L \; 4\sqrt{M\pi} (\sqrt{2}+\sqrt{M}) \sqrt{\frac{7}{8} (M^2+6M+1)}$.
%\end{lemma}

\begin{proof}[Proof of Theorem \ref{NewTheorem}]
%Observe that
Let $\El$ be the event that all $\Llm$ lie in the same interval, $1\le m\le M$, and $\Elc$ its complement, then
\begin{eqnarray}
\be[(\Pl-\Plbar)^2] &=& \be[(\Pl-\Plbar)^2 1_{\El}] +\be[(\Pl-\Plbar)^2 1_{\Elc}]. \nonumber 
\end{eqnarray}
By \eqref{LlMRel} and linearity of $p$ %$\Pl$, $\plb$ 
in each interval, we have
\begin{eqnarray}
\be[(\Pl-\Plbar)^2 1_{\El}] &=& 0. \nonumber 
\end{eqnarray}
%from which the first equation follows.

By Cauchy-Schwartz, we have
\begin{eqnarray}
\be_{|L}[(\Pl-\Plbar)^2 1_{\Elc}] \leq \left(\be_{|L}[(\Pl-\Plbar)^4] \right)^{\frac{1}{2} }
\left(\bp_{|L}[ {\Elc}] \right)^{\frac{1}{2} }, \nonumber
\end{eqnarray}
hence,
\begin{eqnarray}
\be[(\Pl-\Plbar)^2 1_{\Elc}] \leq \be\left[ \left(\be_{|L} [(\Pl-\Plbar)^4] \right)^{\frac{1}{2} } 
\left(\bp_{|L} [ {\Elc}] \right)^{\frac{1}{2} } \right]. \nonumber
\end{eqnarray}
By Lemma \ref{nowelemma}, we have that
\begin{eqnarray}
\left( \be_{|L}[(\Pl-\Plbar)^4] \right)^{\frac{1}{2}} \leq \frac{\sqrt{c_1}}{N_l}, \label{wyzej}
\end{eqnarray}
where $c_1= \frac{7}{8}(M^2+6M+1)$. 

If we denote by $\Bml$ the event that $\Llm$ and $\Ll$ lie in the same interval, then
\[
\Elc = \bigcup_{m=1}^M \Bmlc
\]
and therefore
\begin{eqnarray*}
\bp_{|L}(\Elc) &\le& \sum_{m=1}^M \bp_{|L}(\Bmlc) \ =\ M \, \bp_{|L}(\Blc).
\end{eqnarray*}
By Lemma \ref{thmb2}, this gives
\begin{eqnarray}
\be \left[ \left( \bp_{|L}[{\Elc}] \right)^{\frac{1}{2}} \right] \leq \frac{c_2}{N_l}, \nonumber
\end{eqnarray} 
where $c_2 = c_L \; 4 \sqrt{M \pi} (\sqrt{2}+\sqrt{M})$.
Together with \eqref{wyzej}, we obtain the result. 
\end{proof}

%Theorem \ref{NewTheorem}  follows now directly from (\ref{secmomimp}).

\section{Multilevel tests} \label{MultiTests}

In this section, we present multilevel simulation results based on the estimators from the previous section and illustrating the theoretical findings from there.
We return to the example from Section \ref{Numerical} and estimate expected tranche losses for credit baskets with an increasing number of firms $N_l=M^l$.
%We compare the two estimators from Section \ref{subsec:basic}. % and \ref{subsec:improved}.

For the estimator $Z_l$ from (\ref{mlestisimple}), an upper bound for the variance -- although not a sharp one -- is analytically known from (\ref{basicvar}) and we could use that to determine the number $n_l$ of samples on level $l$ which is
required to bring the variance contribution under a desired threshold.
For the improved estimator $\overline{Z}_l$ from (\ref{mlestiimpr}), however, the bound in (\ref{varbar}) contains the unknown Lipschitz constant of the CDF of $F_L$ via Theorem \ref{NewTheorem}.
%In our study, we know the value of $K$, hence Proposition \ref{ThmMultilevel2} is relevant to us. 
%For the original Multi-level estimator we know the value of $c_2$, therefore we can use the optimal number of simulations given in the in the proof of the Proposition. However, for the improved estimator, the value of $c_2$ is unknown. 
In order to determine the optimal allocation $n_l^*$, we use the following algorithm as per \cite{MikeMultilevel}. 
In contrast to there, the upper level $K$ is fixed here which simplifies the stopping criterion somewhat.
%But in contrast to Giles's algorithm, we do not use any stopping criteria of the method, since we already know when we wish to finish the calculations.

\begin{enumerate}
\item Start with $k=1$.
\item Estimate the variance $V_k$ of a single sample using $n_k=10^4$ realisations.
\item Calculate the optimal number of samples, $n_l^{*}$, for $l=0,1,\ldots,k$, using 
\begin{equation} \label{OptimalNl}
n_l^{*}= \left \lceil{ \gamma^{-2} \sqrt{V_l \; N_l^{-1}} \left( \sum_{j=1}^{k} \sqrt{V_j \; N_j} \right)} \right \rceil,
\end{equation}
where $\gamma^2$ is a chosen upper bound of $Var[G_K]$. % and $V_j$ is the empirical variance of a single sample of $Z_l$ or $\overline{Z}_l$.
%\eqref{OptimalNl}. 
\item Draw extra samples for each level according to $n_l^{*}$.
\item If $k< K$, set $k=k+1$ and go to 2.
\item If $k=K$, finish.
\end{enumerate}

\begin{rem}
As per \cite{MikeMultilevel}, choosing $n_l^{*}$ by \eqref{OptimalNl}, guarantees that the variance $Var[G_K]$ is bounded by $\gamma^2$, since
\begin{equation*}
\begin{split}
Var[G_K] \;=\; \sum_{l=1}^{K} \left(n_l^{*} \right)^{-1} V_l \leq \sum_{l=1}^{K} \left( \gamma^{-2} \sqrt{V_l N_l^{-1}} \sum_{j=1}^{K} \sqrt{V_j N_j} \right)^{-1} V_l \; =\;  \gamma^{2}.
\end{split}
\end{equation*}
%where we used the lower bound of $n_l^{*}$. 
A side effect is that, for $k<K$, the variance is smaller than for $k=K$, since
\begin{equation*}
Var[G_k]=\sum_{l=1}^{k} \left(n_l^{*}\right)^{-1} V_l < \gamma^2 \frac{\sum_{l=1}^{k}\sqrt{V_l N_l}} { \sum_{l=1}^{K} \sqrt{V_l N_l}}.
\end{equation*}
Hence, if we compute estimators $G_k$ for all $k$ as a by-product of $G_K$, the variance is the smallest for $G_1$ and then for $G_k$, $k=2, \ldots, K$, gradually reaches the upper bound $\gamma^{2}$.
This effect can be observed in Figure \ref{fig_Multilevel_Resampling}.D. 
\end{rem}

%We calculate the expected losses in all of the tranches at the same time with the aim to bound the variance of the estimator of the loss in the equity tranche. We choose this tranche since in the preliminary simulations the relative confidence intervals for the equity tranche were the widest. It should be noted that, the conditions stated in Proposition \ref{ThmMultilevel2} are met for all of the considered ways of simulating $X_0$.

In Figure \ref{fig_Multilevel_Resampling}
we show results for the same parameter setting as in Section \ref{Numerical} and only for the equity tranche. Results from other tests were very similar and did not show any noteworthy additional effects.
%The results for the equity tranche can be observed in Figures \ref{fig_Multilevel_x0const} and \ref{fig_Multilevel_Resampling} for the original Multilevel estimator, and also in Figure \ref{fig_Multilevel_Resampling_Improved} for the improved Multilevel estimator. 
In order to easily see the rate of convergence in \ref{fig_Multilevel_Resampling}.A., we plot the logarithm of $V_l$ to base $M$, together with %$f_k$ given as 
\begin{equation} \label{hk}
f_k= -\beta \; k+ f_0
\end{equation}
for different values of $\beta$.
%For comparison, in the case of the improved estimator we plot both a trend with $\beta=\frac{3}{2}$ and with $\beta=1$. 
%In both cases when $X_0$ is constant or drawn from a Normal distribution,
The estimated slope is $\widehat{\beta} \approx 1$ for the original estimator and
$\widehat{\beta} \approx {3}/{2}$ for the improved estimator, which agrees with the theoretical findings.
The order of convergence of $\vert E[\Pl-\Plminusone] \vert$ is  $\widehat{\alpha} \approx 1$, which also agrees with the previous results. 
As can be observed in Figure \ref{fig_Multilevel_Resampling}.C, he number of samples ranges from
$150$ millions for $k=1$ to $34 000$ for $k=7$.
The improved estimator gives further reductions in computational time: the total number of samples ranges now from $35$ millions for $k=1$ to only $350$ for $k=7$.
The standard deviation of $G_k$ is an increasing function of $k$, and is less than or equal to the chosen upper bound $\gamma = 4 \times 10^{-6}$.

%When $X_0$ is not resampled, similarly to the previous cases, $\beta \approx 1$. The total number of simulations used to obtain the results is also of a similar order: $200$ millions for $k=1$ and $2$ thousand for $k=9$. In line with the previous empirical results, $\hat{\alpha} \approx \frac{1}{2}$.

%For the improved estimator, the estimated slope $\widehat{\beta}$ is very close to $\frac{3}{2}$, which agrees with our theoretical results. 

%\begin{sidewaysfigure}
%\centering
%\scalebox{0.5}
%{\includegraphics{Multilevel_x0const_new.eps}}
%\caption{Multi-level results for the expected loss in the equity tranche of a CDO basket consisting of $N_k$ companies, $N_k=M^k$, $k=1, \ldots, K=9$, $M=5$,
%for the case when $X_0^i=\mu_{X_0}=4.6$ $\forall i$.
%A. Variance of a single Monte Carlo sample, $V_l$, defined in \eqref{VarYl}, together with a fitted trend, $h_l$, given by \eqref{hk}, where $\beta=1$. B. The estimator on level $l$, $Y_l$, defined in \eqref{Yl}, its standard deviation, which is the square root of \eqref{VarYl}, and a fitted trend, $y_l$, defined by \eqref{yk}, where $\hat{\alpha}=1$. C. Optimal number of simulations, $n_l^*$, used in each level, calculated according to  \eqref{OptimalNl} for $k=K$. D. Standard deviation of $G_K$, defined in \eqref{VarY}, and the upper bound of it, $\gamma$, chosen for the experiment.
%}
%\label{fig_Multilevel_x0const}
%\end{sidewaysfigure}

\begin{figure}
\centering
\scalebox{0.6}
{\hspace{-2.5 cm} \includegraphics{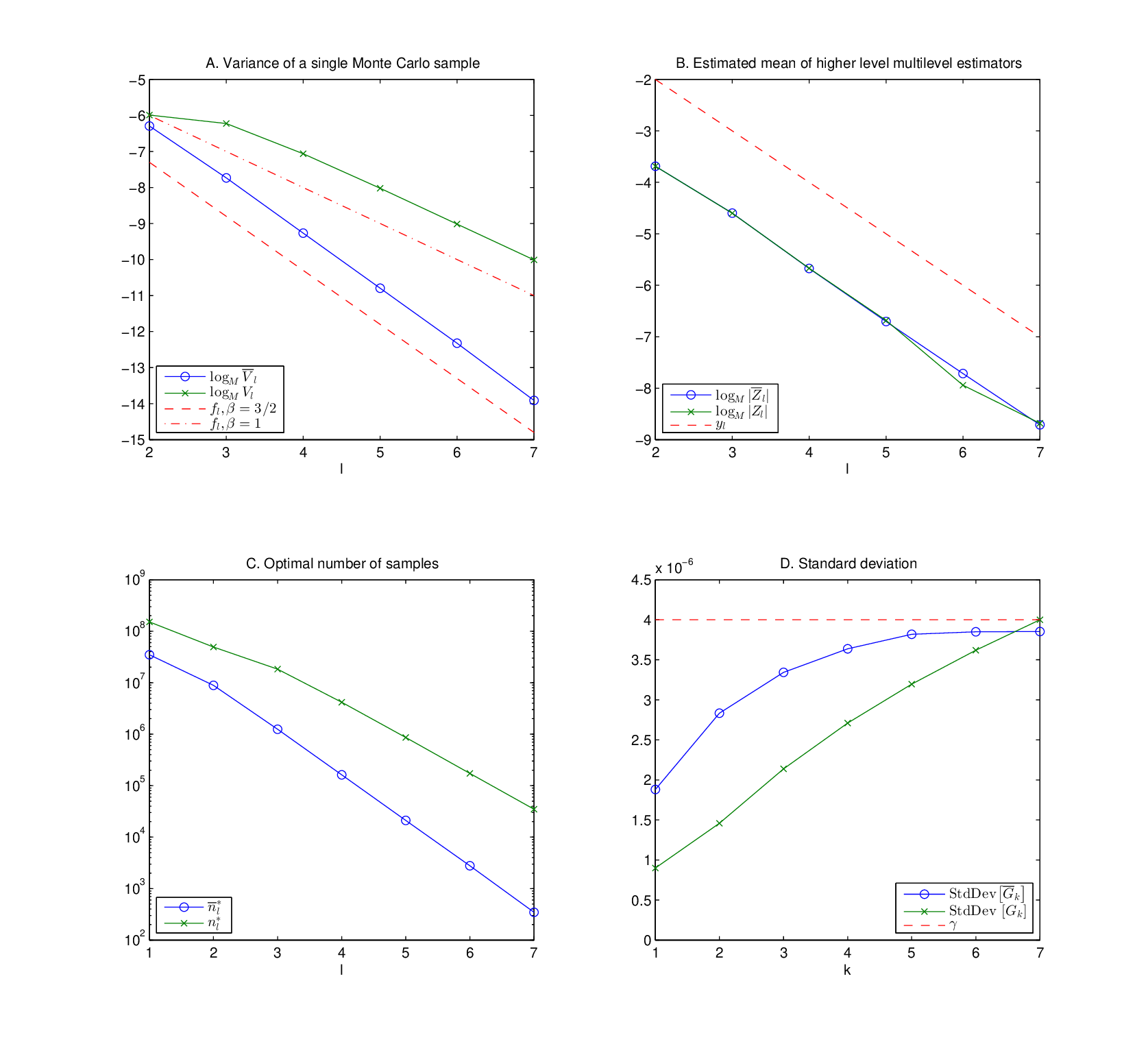}}
\caption{
Multilevel results for the expected loss in the equity tranche of a CDO basket consisting of $N_k$ companies, $N_k=M^k=5^k$, $k=1, \ldots, 7$. %, $\mu_{X_0}=4.6$, $\sigma_{X_0}=0.8$, $M=5$, $K=7$.
Overlined quantities refer to the estimator $\overline{Z}_l$ from (\ref{mlestiimpr}), all others to the standard estimator $Z_l$ from (\ref{mlestisimple}).
A.~Variance of a single Monte Carlo sample, $V_l$ and $\overline{V}_l$, together with a predicted trend, $f_l$, given by \eqref{hk}, where $\beta=1$ or $\beta=3/2$.
B.~Mean at level $l$, $Z_l$ and $\overline{Z}_l$, and a trend, $y_l$, defined by \eqref{yk}, with slope -1. C.~Optimal number of simulations in both cases, $n_l^*$ and $\overline{n}_l^*$,
calculated according to  \eqref{OptimalNl} for $k=K=7$. D.~Standard deviation of multilevel estimators $G_k$  defined in \eqref{multiest}, and similar for $\overline{G}_k$, with their chosen upper bound, $\gamma$.}
\label{fig_Multilevel_Resampling}
\end{figure}

%\begin{sidewaysfigure}
%\centering
%\scalebox{0.5}
%{\includegraphics{PaperConvergenceMulti_resampling_Improved2.eps}}
%\caption{Multi-level resutls for the improved estimator of the expected loss in the equity tranche for the case when $\forall i$ $X_0^i \sim N(\mu_{X_0}, \sigma_{X_0})$ and $X_0^i$ is resampled for each simulation of $(X_t^i)_{0 \leq t \leq T}$, a CDO basket consists of $N_k$ companies, $N_k=M^k$, $k=1, \ldots, K$, $\mu_{X_0}=4.6$, $\sigma_{X_0}=0.8$, $M=5$, $K=7$. A. Variance of a single Monte Carlo path, $V_l$, defined in \eqref{VarYl}, together with a fitted trend, $h_l$, given by \eqref{hk}, where $\hat{\beta}=1.5$, and also with a trend $s_l$, where $\hat{\beta}=1$. B. The estimated loss of a level, $Y_l$, defined in \eqref{Yl}, its standard deviation, which is a square root of \eqref{VarYl}, and a fitted trend, $y_l$, defined by \eqref{yk}, where $\hat{\alpha}=1$. C. Optimal number of simulations, $n_l^*$, used in each level, calculated according to  \eqref{OptimalNl} for $k=K$, together with optimal number of simulations $n_l^{Original}$ obtained for the original Multi-level estimator for the same settings. D. Standard deviation of $G_K$, defined in \eqref{VarY}, and the upper bound of it, $\gamma$, chosen for the experiment.}
%\label{fig_Multilevel_Resampling_Improved}
%\end{sidewaysfigure}

\section{Conclusions and extensions} \label{Conclusion}

A main focus of this paper was the construction of an efficient simulation algorithm for functionals of a large number of exchangeable random variables. 
For a specific set-up,
we were able to demonstrate optimal complexity order by theoretical analysis and numerical illustrations.

\subsection*{Discussion}

The results from the previous section show that the computational savings can be significant in situations of 
practical relevance. As seen from Figure \ref{fig_Multilevel_Resampling}.C, already for $N=125$ (i.e., $k=3$), 
the size of a CDO basket, the required number $n_3$ of samples on this level is reduced by about two orders 
of magnitude compared to the number of samples for $k=1$, $n_1$. It is roughly this number which would be 
required for a standard (i.e., single level) estimator on level $3$ for a variance comparable to the one achieved 
by the multilevel estimator at substantially lower cost.

\subsection*{Extensions -- random recovery and random factor loadings}

There is ample empirical evidence that a basic factor model such as the one described in Section \ref{Numerical} 
does not adequately reproduce observed market spreads of credit derivatives and other stylised facts of credit 
markets. Two effects that have so far been neglected are credit contagion (i.e., the default of one firm has an 
impact on the credit worthiness and dependence structure of others) and the dependence of recovery values 
on the wider credit environment.
We focus here on the latter effect and follow \cite{AndersenSidenius} for a model that captures this dependence.

Consider thus the total loss as given by
\[
\widetilde{L}_N = \frac{1}{N} \sum_{i=1}^N l_i \, Y_i,
\]
where $l_i$ represent a random \emph{loss-given-default} for company $i$ and $Y_i$ are default indicators as previously. % and are i.i.d.\ conditional on $\FL$
It is sometimes convenient to write
\[
l_i = l_{\max} \cdot \left(1-R_i\right),
\]
where $l_{\max}$ is a (constant) notional maximum loss and $R_i \in [0,1]$ is the (random) recovery rate of the $i$-th 
firm. In keeping with our general framework, we assume that the $R_i$ are identically distributed and independent 
conditional on $\FL$.

For continuous payoffs $p$, we still have
\begin{eqnarray} \label{lossInftyTilde}
\widetilde{L}_N &\rightarrow& \be[\widetilde{L}_N] =: \widetilde{L} \quad \text{for } N \rightarrow \infty, \;\; \bp_{|L}-a.s., \\
\widetilde{P}_N := p(\widetilde{L}_N) &\rightarrow& \be[\widetilde{P}_N] =: \widetilde{P} \quad \text{for } N \rightarrow \infty, \;\; \bp_{|L}-a.s.
\label{payInftyTilde}
\end{eqnarray}

%\begin{eqnarray*}
%bla
%\end{eqnarray*}

The $L^2$-convergence is described in the following.
\begin{corollary}[to Theorem  \ref{NewerTheorem}]
 \label{NewCorollary}
Let $\widetilde{P}$ and $\widetilde{P}_N$ be given by (\ref{payInftyTilde}),
and assume that $p$ is Lipschitz with constant $c_p$.
We have that
\begin{eqnarray}
\label{alphapart2}
\vert \be[\widetilde{P}_N -\widetilde{P}] \vert & \leq & \frac{c_p \ l_{\max}}{\sqrt{N}},\\ 
Var[\widetilde{P}_N - \widetilde{P}] & \leq &  \frac{c_p^2 \ \l_{max}^2}{N}.
\label{betapart2}
\end{eqnarray}
\end{corollary}
\begin{proof}
In the same way as the proof of Lemma \ref{OldLemmaDwa},
\begin{eqnarray*}
\be_{|L}[(\widetilde{P}_N-\widetilde{P})^2]  \leq  c_p ^2 \ \be_{|L}[(\widetilde{L}_N-\widetilde{L})^2] 
= c_p^2 \ Var[\widetilde{L}_N|\fm]
= \frac{c_p^2}{N} \ Var[l_i \Yi|\fm] 
\le \frac{c_p^2 \ l_{\max}^2 }{N},
\end{eqnarray*}
which gives the result for the variance. The result for the expectation follows again immediately.
\end{proof}

The order $1/2$ for the convergence of the expectations and order $1$ for the variances is sufficient to be able to 
apply Corollary \ref{complex1} to establish the $\epsilon^2 \log^2\epsilon$ complexity for MSE $\epsilon^2$ of 
the multilevel method. The following numerical tests indicate that the order 1/2 is not sharp and indeed we expect 
order 1 for sufficient regularity of the payoffs and/or distribution function of $L$. The proof of this becomes more 
technical than in the pure Bernoulli case because we lose the explicit form of the characteristic function. Thus, and 
because of the irrelevance of this for the convergence speed of the multilevel method, we do not pursue this further 
here.

We now consider a particular model similar to the one in \cite{AndersenSidenius} and give a numerical illustration.
Specifically, let
\begin{eqnarray}
\label{xproc}
X_t^i &=& X_0^i+\beta t+\sqrt{1-\rho} \; W_t^i+ \sqrt{\rho} \; B_t + J_t , \quad t>0, \\
R_t^i &=& \Phi(\mu^R + \beta^R t + \sigma^R \sqrt{1-\rho^R} \; \xi_t^i+  \sigma^R \sqrt{\rho^R} \; B_t + J_t),
\label{rproc}
\end{eqnarray}
where the processes in the first line are defined as in (\ref{Xti}), and in the second line $(\xi^i)$ is a standard 
Brownian motion independent of everything else, while $\mu^R, \beta^R$, $\sigma^R>0$ and $0\le\rho^R\le 1$ 
are constants. $\Phi$ is the cumulative density of the standard normal, but could be replaced by any increasing 
function $f:\mathbb{R} \rightarrow [0,1]$. This has the effect that the recovery rate is positively correlated with 
the market factors $B$ and $J$, with some idiosyncratic noise, and thus there is a negative dependence between 
recovery rates and default frequencies. See, for instance, \cite{AltmanEtAl} for an early but influential study of 
this empirical fact.

The above model is not precisely contained in the set-up of Corollary \ref{NewCorollary}, because the recovery rates processes (\ref{rproc}) are not independent conditional on $L$ (as a result of the different exposure of $R^i$ to $B$ and $J$ compared to $X^i$). However, both the $X^i$ and $R^i$ are independent conditional on $B$ and $J$, and therefore if $\FL$ in the proof of Corollary \ref{NewCorollary} is replaced by $\mathcal{F}_B\times\mathcal{F}_J$, the filtration generated by the common factors (a larger filtration than $\FL$), the result still follows.

%Since $\FL$ is identical to the filtration generated by the market factors $B$ and $J$, the above convergence result holds.

\bigskip

In the numerical simulations, we choose the values of $\mu_R$, $\beta_R$ and $\sigma_R$ such that, for all $i=1, \ldots,N$, $R_0^i=0.4$,  $\be[l_i \, Y_i \, \vert \, \tau^i \leq T]  \approx 0.7$ and $Var [ l_i \, Y_i \, \vert \, \tau^i \leq T] \approx 0.16$, compared to Section \ref{MultiTests}, where the recovery rate is constant at $0.4$. In particular, we have $\mu_R=-0.25$, $\beta_R=1.5$ and $\sigma_R=0.9$.
Also, we assume that $\rho_R=\rho$. All other parameters are the same as in the tests in Sections 
\ref{Numerical} and \ref{MultiTests} (see the paragraphs after (\ref{Xti}) for the model set-up and parameter values).

\begin{figure}
\centering
\scalebox{0.6}
{\hspace{-2.5 cm} \includegraphics{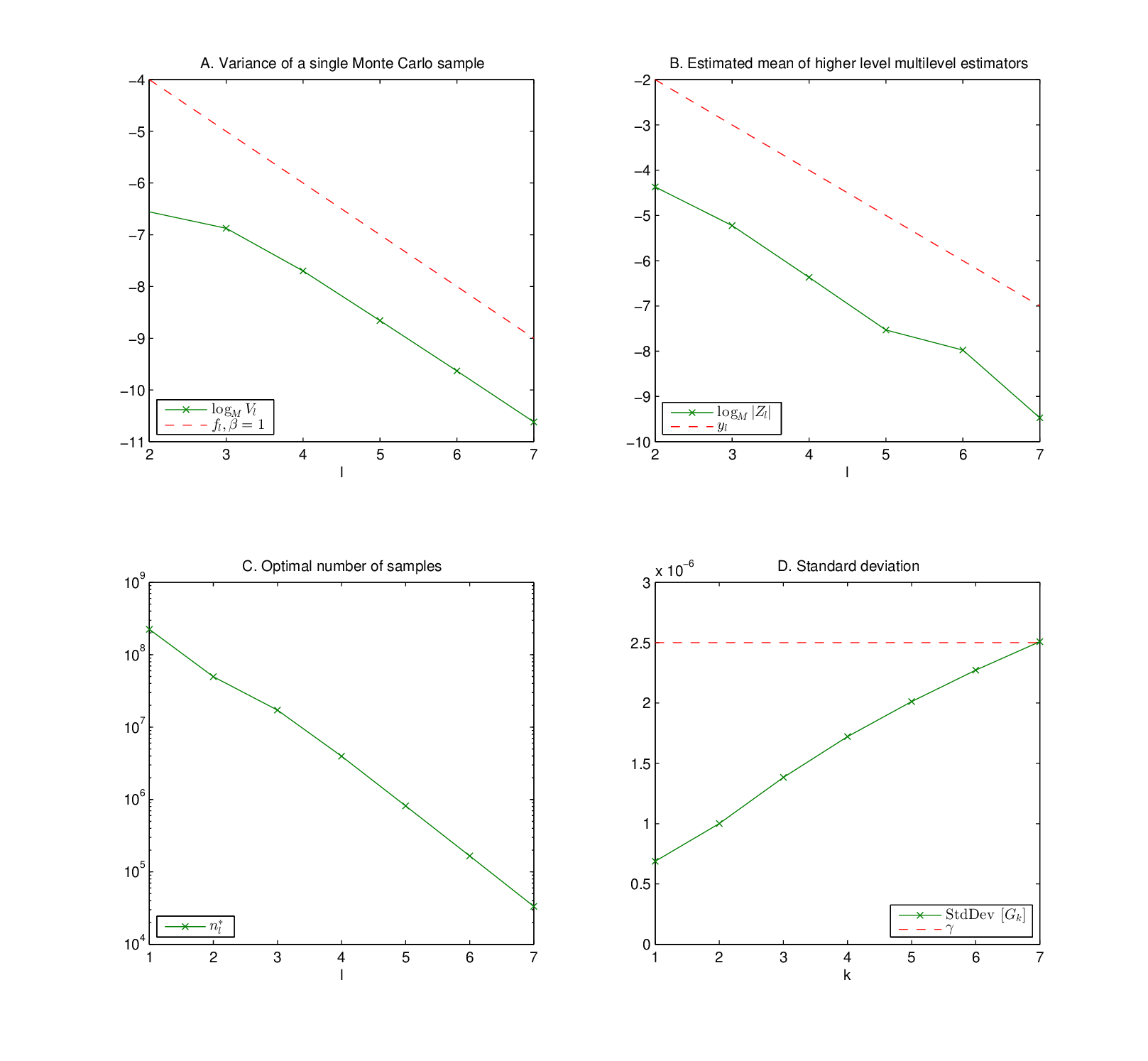}}
\caption{
Multilevel results for the expected loss in the equity tranche of a CDO basket consisting of $N_k$ companies, $N_k=M^k=5^k$, $k=1, \ldots, 7$, where the recovery rate is random and given by \eqref{rproc}. The quantities refer to the standard estimator $Z_l$ from (\ref{mlestisimple}).
A.~Variance of a single Monte Carlo sample, $V_l$ together with a predicted trend, $f_l$, given by \eqref{hk}, where $\beta=1$.
B.~Mean at level $l$, $Z_l$ and a trend, $y_l$, defined by \eqref{yk}, with slope -1. C.~Optimal number of simulations, $n_l^*$,
calculated according to  \eqref{OptimalNl} for $k=K=7$. D.~Standard deviation of multilevel estimators $G_k$  defined in \eqref{multiest} with chosen upper bound, $\gamma$.
}
\label{fig_Multilevel_Resampling_RR_Trend_Final}
\end{figure}

The results in Fig.~\ref{fig_Multilevel_Resampling_RR_Trend_Final} are presented in the same format as Fig.~\ref{fig_Multilevel_Resampling} earlier for the constant recovery rate. There is clear evidence that the convergence of the variance and mean are still both of first order in $1/N$, where $N$ is the basket size.

\bigskip

Another extension, also proposed in \cite{AndersenSidenius}, are \emph{random factor loadings} of the type
\begin{eqnarray}
X_t^i &=& X_0^i+\beta t+ \nu \; W_t^i+ a(B_t) \; B_t, \quad t>0,
\end{eqnarray}
where $a$ is a given deterministic function and $\nu^2 = 1-Var(a(B_t) B_t)$. This model can capture contagion 
effects where, for decreasing $a$, firm values are more closely correlated to the common market factor in bad 
times. Such an extension fits directly in the general framework developed earlier in this paper, assuming the 
technical conditions on the cumulative density $F_L$ hold where needed (see in particular Theorem \ref{OldTheorem}).

\subsection*{Outlook}

We would expect there to be scope to apply the presented nested simulation approach to a wider range of 
settings beyond the particular application studied here.
An interesting extension would be to the model from \cite{GieseckeEtAl12}, where the 
analysis requires further tools accounting especially for the heterogeneity of the basket, resulting in non-exchangeability.
While our motivation comes from credit baskets and some of the later results are specific to piecewise linear 
functionals encountered in the valuation of basket credit derivatives, we hope there to be a wider relevance of the 
main approach to the simulation of certain functionals arising in large interacting particle systems and elsewhere.

\bigskip
\bigskip

\noindent
{\bf Acknowledgement:} We thank Mike Giles for suggesting the improved estimator, and two anonymous referees for their helpful suggestions
which improved the presentation of the results. 

%jak dziala bibilio, w mainie Tools\Bivbtex potem dwa razy Tools\QuickBuild
\bibliographystyle{siam}

\appendix
\section{Moment computations}
\label{app:moments}
\begin{proof}

[of Lemma \ref{nowelemma}]
We begin by showing (\ref{fourtha}) and then deduce (\ref{fourthb}) and (\ref{fourthc}).
We have %$P_l$ is a Lipschitz function, we get
\[
\mid \Pl-P \mid \;  \leq  \; \mid \Ll-L \mid,
\]
where
\begin{eqnarray}
\Ll= \frac{1}{N_l}\sum_{i=1}^{N_l} \Yi.  \nonumber 
\end{eqnarray}
%This leads to
Hence, we get
\begin{equation*} 
\be_{|L} [(\Pl-P)^4] \leq \be_{|L} [(\Ll-L)^4] = \be_{|L} \bigg[ \bigg(\frac{1}{N_l} \sum_{i=1}^{N_l} (Y_i-L)\bigg)^4\bigg].
\end{equation*}
As $L$ is $\fm$-measurable and the $Y_i$ are independent and identically distributed given $\fm$
with $\be_{|L}[Y_i-L] = 0$, we have
\begin{eqnarray*}
\be_{|L} [(\Ll-L)^4] &=& \frac{1}{N_l^4} \be_{|L} \bigg[ \sum_{i=1}^{N_l} (Y_i-L)^4 +
6\sum_{i\neq j} (Y_i-L)^2(Y_j-L)^2 \bigg] \\
&=& \frac{1}{N_l^3} \left( (1-L)^4L+L^4(1-L)\right)+ \frac{3(N_l-1)}{N_l^3} \left((1-L)^2L+L^2(1-L)\right)^2 \\
&=& \frac{3L^2(1-L)^2}{N_l^2} + \frac{L(1-L)(1-6L(1-L))}{N_l^3}.
\end{eqnarray*}
Using the fact that $0\leq L(1-L) \leq 1/4$ we have the required bound in (\ref{fourtha}).

For (\ref{fourthb}), observe that there are many ways of estimating this fourth moment; we choose 
the following
\begin{eqnarray}
(\Pl-\Plminusone)^4 &=& \left((\Pl-P)-(\Plminusone-P)\right)^4 \nonumber \\
&\leq & 2 \; \left( (\Pl-P)^4 + (\Plminusone-P)^4 \right) + 12 (\Pl-P)^2 (\Plminusone-P)^2. \nonumber
\end{eqnarray}
Therefore, using Cauchy-Schwarz on the last term and applying (\ref{fourtha}) we have
\begin{eqnarray*}
\be_{|L}[(\Pl-\Plminusone)^4] &\leq & 2 \left( \be_{|L} [(\Pl-P)^4] + \be_{|L} [(\Plminusone-P)^4]
\right) \\
& & \qquad \qquad + \; 12 \; \be_{|L} [(\Pl-P)^4]^{1/2} \, \be_{|L} [(\Plminusone-P)^4] ^{1/2}
  \\
&\leq & \frac{7}{8N_l^2} +  \frac{7}{8N_{l-1}^2} + \frac{42}{8N_lN_{l-1}}
\end{eqnarray*}
as required to obtain (\ref{fourthb}).

Finally, (\ref{fourthc}) follows from
\begin{eqnarray*}
\be_{|L} [(\Pl-\Plbar)^4] &=& \be_{|L} \bigg[ \bigg( \frac{1}{M} \sum_{m=1}^{M} (\Pl-\Plm)  \bigg)^4 \bigg] \\
&\le& \be_{|L} \Big[\frac{1}{M} \sum_{m=1}^{M} (\Pl-\Plminusone)^4\Big] \\
&=& \be_{|L} [(\Pl-\Plminusone)^4],
\end{eqnarray*} 
and an application of (\ref{fourthb}).
\end{proof}

\end{document}